\documentclass{amsart}
\usepackage{latexsym}
\usepackage{amsfonts,amsmath,amssymb,stmaryrd,amscd,amsxtra,amsthm,verbatim,calc}
\usepackage[all]{xy}

\RequirePackage[dvipsnames,usenames]{color}

\theoremstyle{plain}
\newtheorem{thm}{\bf Theorem}[section]
\newtheorem{prop}[thm]{\bf Proposition}
\newtheorem{lem}[thm]{\bf Lemma}
\newtheorem{conj}[thm]{\bf Conjecture}
\newtheorem{cor}[thm]{\bf Corollary}

\theoremstyle{definition}
\newtheorem{defn}[thm]{\bf Definition}
\newtheorem{rem}[thm]{\bf Remark}

\newtheorem{ex}[thm]{\bf Example}
\newtheorem{defn-prop}[thm]{\bf Definition-Proposition}
\newtheorem{question}[thm]{\bf Question}

\DeclareMathOperator{\Hom}{Hom}
\DeclareMathOperator{\Min}{Min}

\DeclareMathOperator{\Ann}{Ann}

\DeclareMathOperator{\Supp}{Supp}

\DeclareMathOperator{\im}{Im}

\DeclareMathOperator{\cL}{\mathcal{L}}

\DeclareMathOperator{\cS}{\mathcal{S}}

\newcommand{\ro}{\hskip .6 pt\rule[4.6 pt]{3.5 pt}{3.5 pt}}
\newcommand{\rob}[1]{^{\boxempty^{#1}}}
\newcommand{\roi}{^{\boxempty}}
\newcommand{\red}{_{\text{red}}}

\DeclareMathOperator{\Q}{Q}
\DeclareMathOperator{\hgt}{ht}
\DeclareMathOperator{\supht}{supht}
\DeclareMathOperator{\rk}{rank}
\DeclareMathOperator{\mdim}{mdim}
\DeclareMathOperator{\embdim}{embdim}
\DeclareMathOperator{\fa}{\mathfrak{a}}
\DeclareMathOperator{\fb}{\mathfrak{b}}
\DeclareMathOperator{\fh}{\mathfrak{h}}
\DeclareMathOperator{\fm}{\mathfrak{m}}
\DeclareMathOperator{\fM}{\mathfrak{M}}

\DeclareMathOperator{\fN}{\mathfrak{N}}

\DeclareMathOperator{\fp}{\mathfrak{p}}
\DeclareMathOperator{\fP}{\mathfrak{P}}
\DeclareMathOperator{\fq}{\mathfrak{q}}
\DeclareMathOperator{\ux}{\underline{x}}

\DeclareMathOperator{\MaxSpec}{MaxSpec}

\DeclareMathOperator{\lm}{lim}

\def\vect#1#2{{#1}_1, \, \ldots, \, {#1}_{#2}}
\def\inc{\subseteq}
\begin{document}

\title{Content of Local Cohomology, Parameter Ideals, and Robust Algebras}
\author{Melvin Hochster and Wenliang Zhang}
\address{Department of Mathematics, University of Michigan, Ann Arbor, MI 48109, USA}
\email{hochster@umich.edu}
\address{Department of Mathematics, Statistics and Computer Science, University of Illinois, Chicago, IL 60607, USA}
\email{wlzhang@uic.edu}
\thanks{M.H. is partially supported by the National Science Foundation through grant DMS\#1401384 and W.Z. is partially supported by National Science Foundation through grant DMS \#1606414}
\numberwithin{equation}{thm}
\begin{abstract} This paper continues the investigation of quasilength, of content of local cohomology with
respect to generators of the support ideal,
and of robust algebras begun in joint work of Hochster and Huneke.
We settle several questions raised by Hochster and Huneke. 
In particular, we give a family of examples of top local cohomology modules both in equal characteristic
0 and in positive prime characteristic that are nonzero but have content 0.  We use the notion of
a robust forcing algebra (the condition turns out to be strictly stronger than the notion of a solid forcing algebra in,
for example, equal characteristic 0) 
to define a new closure operation on ideals.  We prove that this new notion of closure
coincides with tight closure for ideals in complete local domains of positive characteristic, which
requires proving that forcing algebras for instances of tight closure are robust, and study several related problems.
This gives, in effect, a new characterization of tight closure in complete local domains of positive characteristic. As a byproduct, we also answer a question of Lyubeznik in the negative.
\end{abstract}   

\maketitle
\pagestyle{myheadings}
\markboth{\scriptsize{\textsc{MELVIN HOCHSTER AND WENLIANG ZHANG}}}
{\scriptsize{\textsc{CONTENT OF LOCAL COHOMOLOGY, PARAMETER IDEALS, AND ROBUST ALGEBRAS}}}

\section{Introduction}\label{intro}   

This paper continues the study of the notions of quasilength, content,
and robustness initiated in  \cite{quasi-length} and resolves questions raised in that
paper.  We use these ideas to give a new characterizaton of tight closure in equal
characteristic $p>0$.  The condition we use is similar to the characterization of tight
closure in \cite{solid-closure} using solid closure, but the condition we impose on the forcing
algebras that arise is {\it a priori} stronger than being solid.  
The stronger condition, which is that the
forcing algebra be {\it robust} (defined briefly two paragraphs below, and with more
detail in  Definition~\ref{defn-robust-algebra}), nonetheless gives the usual notion of tight 
closure for ideals in complete local domains. 

There is great interest in extending the notion of tight closure to rings of mixed
characteristic.  An example of Paul Roberts \cite{Roberts-Example} shows that solid closure
is not the right notion, in that, even in a regular local ring of dimension 3 in equal
characteristic 0, it is not true that every ideal is solidly closed.  Roberts proves
this by showing that a certain forcing algebra is solid, when one hopes it should
not be. Specifically, let $k$ be a field of characteristic 0, let $A$ be the formal power series
ring $K[[x_1,\, x_2,\, x_3]]$, and let $R = A[z_1,\, z_2, \, z_3]/(g)$
where $g = x_1^2x_2^2x_3^2 - \sum_{i=1}^3 z_ix_i^3$.  $R$ is a generic forcing algebra over
for the element $x_1^2x_2^2x_3^2$  and the ideal $(x_1^3,\,x_2^3,\,x_3^3)A$. 
Roberts shows that $H^3_{(x_1,x_2,x_3)}(R) \not=0$, so that $R$ is a solid $A$-algebra:  see \S\ref{robcl} and
Definition~\ref{defsolcl} in particular. 
However, this algebra is {\it not} robust ({\it cf.} \cite[Example 3.11]{quasi-length} and 
Definition \ref{defn-robust-algebra}).  It remains an intriguing open question to determine the content $c$
of $H^3_{(x_1,x_2,x_3)}(R)$ with respect to $x_1, \, x_2, \,x_3$.  The fact that $g$ vanishes in $R$
implies at once that $c \leq 26/27$, but so far as the authors know, $c$ might be 0.

In \S\ref{prel} we review the needed background concerning quasilength, content,
and robust algebras.  In \S\ref{par} we study quasilength and content for modules
over a local ring with respect to a  system of parameters.  We give some surprising
examples of the failure of additivity on direct sums.  However, over an equicharacteristic local domain of Krull dimension
$d$,  it is true that the content of the $d\,$th local cohomology module of  a module of torsion-free rank $r$
with respect to a system of parameters is $r$. Theorem~\ref{main-thm} gives a more general statement.

A sequence of elements $\vect x d$ generating a proper ideal of a ring $S$
is called a Q-{\it sequence} if for every positive integer $t$, any finite filtration of 
$S/(x_1^t, \, \ldots, \, x_d^t)$ with factors
that are quotients of $S/(\vect x d)S$ requires at least $d^t$ factors.
An algebra $S$ over a local ring
$(R,\, m, \,K)$ of Krull dimension $d$ to be {\it robust} if every system of 
parameters $\vect x d$ of $R$ is Q-sequence in $S$. 
That $S$ be robust implies that $H^d_m(S) \not= 0$.  The condition that 
$H^d_m(S) \not= 0$ is equivalent
to the condition that $S$ be solid if $R$ is complete ({\it cf.} \cite[Corollary 2.4]{solid-closure}).

In \S\ref{notrob} we show that there are non-zero local cohomology modules whose content is 0 in all 
characteristics, at the same time answering a question raised in 
\cite[Example 3.14]{quasi-length}.  Specifically,  let $S = A[x,y,u,v]$ be a polynomial 
ring in 4 variables over a Noetherian commutative ring $A$, and let $R = A[xu, yv, xv+yu]$. Then $H^3_{(xu,yv,xv+yu)}(S)\neq 0$ and hence $S$ is a solid $R$-algebra.  In  \cite{quasi-length},  it was asked 
whether $xu, \, yv,\, xv+yu$ is a Q-sequence.  It is shown here that the content with respect to $xu, \, yv,\, xv+yu$
 is 0 in all characteristics; if $xu, \, yv,\, xv+yu$ were a Q-sequence, the content would be 1. 

\S\ref{robcl} gives the definition of robust closure, and in \S\ref{charp} we prove that it agrees
with tight closure in complete local domains of characteristic $p >0$. This is equivalent to the following statement,
which is one of our main results.

\begin{thm}[Theorem \ref{solid-implies-content-1}]
Let $R$ be a local ring of prime characteristic $p >0$.  Suppose that  
$I = (\vect f h)$ is an ideal of $R$,  that $g \in R$, that 
$g \in I^*$,  the tight closure of $I$,  and  that $$S = R[\vect Z h]/(g - \sum_{i=1}^h f_iZ_i)$$ 
(which is a generic forcing algebra for $(R,\, (f_1,\, \ldots,\, f_h),\,g)$: see Definition~\ref{defforce}).  Then  $S$ is a robust $R$-algebra. 
\end{thm}

\S\ref{Qsup} discusses the relationship between being a Q-sequence and superheight. 
In  a Noetherian ring $R$, whether $\vect x d$ is Q-sequence is related to whether
the height of $(\vect xd)S$ becomes $d$ in some Noetherian $R$-algebra $S$.
In equal characteristic, the latter condition is sufficient for $\vect x d$ to be a
Q-sequence.  We show, however, that it is not necessary. The examples are subtle,
and the proofs depend on difficult theorems. Our examples also answer negatively a question of 
Gennady Lyubeznik (page 144 in \cite{Lyubeznik-survey}) on the vanishing of local cohomology modules, {\it cf.} Proposition \ref{prop: forcing-algebra-cubic}.  


\S\ref{qu} describes some conjectures and questions that are related to the results
of this paper.

\section{Preliminaries}\label{prel}  
In this section we collect some basic facts about content of local cohomology from \cite{quasi-length}. 
We begin with the definition of {\it quasilength}.

\begin{defn}[Quasilength]
Let $R$ be a commutative ring (not necessarily Noetherian), $I$ a finitely generated ideal of 
$R$, and $M$ an $R$-module.\par 
$M$ is defined to have {\it finite $I$-quasilength} if there is a finite filtration of $M$ in which the 
factors are cyclic modules killed by $I$.\par
The $I$-{\it quasilength} of $M$, denoted by $\cL_I(M)$, is defined to be the minimum number of factors in such a filtration. If $M$ does not have finite $I$-quasilength, then its $I$-quasilength is defined to be $\infty$.
\end{defn}

One can check that $M$ has finite $I$-quasilength if and only if $M$ is finitely generated and 
is killed by a power of $I$, and $\cL_I(M)$ is bounded below by the least number of generators of $M$.
({\it cf.} Propostion 1.1(a) of \cite{quasi-length}.) 

Assume $R$ is a Noetherian commutative ring and $M$ a finitely generated $R$-module. Let $x_1,\dots,x_d$ be elements of $R$ and $I=(x_1,\dots,x_d)$. We will use $\underline{t}$ to denote the $d$-tuple of positive integers $(t_1,\dots,t_d)$,  
$\underline{t}+k$ to denote the $d$-tuple $(t_1+k,\dots,t_d+k)$ and $I_{\underline{t}}$ to denote 
$(x^{t_1}_1,\dots,x^{t_d}_d)$. One can define
\[(I_{\underline{t}}M)^{\lm}=\bigcup^{\infty}_{k=0}((I_{\underline{t}+k}M):_M(x_1\cdots x_d)^k).\]
If we set $\underline{k}=(k_1,\dots,k_d)\in \mathbb{N}^d$, then one also has that
\[(I_{\underline{t}}M)^{\lm}=\bigcup_{\underline{k}\in\mathbb{N}^d}((I_{\underline{t}+k}M):_M(x^{k_1}_1\cdots x^{k_d}_d)).\] 
As we will see $(I_{\underline{t}}M)^{\lm}$ is closely related to the local cohomology module $H^d_I(M)$. It is well-known that $H^d_I(M)=\varinjlim_{\underline{t}}M/I_{\underline{t}M}$ where in the direct system the map $M/I_{\underline{t}M}\to M/I_{\underline{t}+\underline{k}}M$ is induced by multiplication by $x^{k_1}_1\cdots x^{k_d}_d$ on the numerators. It follows 
that $(I_{\underline{t}}M)^{\lm}$ is the kernel of the canonical map $M\to M/I_{\underline{t}}M\to H^d_I(M)$. Therefore
\[H^d_I(M)=\varinjlim_{\underline{t}}M/(I_{\underline{t}}M)^{\lm}\]
in which all maps in the direct system are injective and each $M/(I_{\underline{t}}M)^{\lm}$ in this direct system 
can be viewed as a submodule of $H^d_I(M)$.\par
Now we are in position to introduce the $\underline{\fh}$-{\it content} of $H^d_I(M)$ {\it with respect to} 
$\underline{x}$ (\cite[page 9]{quasi-length}). 

\begin{defn} 
With $\underline{t}\geq s$ for $s\in \mathbb{N}$ meaning that every $t_j\geq s$, we define\[\underline{\fh}^d_{\underline{x}}(M)=\lim_{s\to \infty}\inf\{\frac{\cL_I(M/(I_{\underline{t}}M)^{\lm})}{t_1\cdots t_d}|\underline{t}\geq s\}.\]
\end{defn}
The existence of $\underline{\fh}^d_{\underline{x}}(M)$ is guaranteed by \cite[Proposition 2.1]{quasi-length}.\par
\[\fh^d_{\underline{x}}(M)=\lim_{s\to \infty}\inf\{\frac{\cL_I(M/(I_{\underline{t}}M))}{t_1\cdots t_d}|\underline{t}\geq s\}.\]
If $\fh^d_{\underline{x}}(R)=1$, then according to \cite[Theorem 3.8]{quasi-length} one also has $\underline{\fh}^d_{\underline{x}}(R)=1$; in this case $\underline{x}$ is called a $\Q$-{\it sequence} and $H^d_I(R)$ is called {\it robust} for $\underline{x}$.\par
By \cite[Theorem 5.1]{quasi-length}, if $R$ is a local ring and $x_1,\dots,x_d$ is a system of parameters, then one has $\underline{\fh}^d_{\underline{x}}(M)=\fh^d_{\underline{x}}(M)$. This is the reason we will mainly 
work with $\fh^d_{\underline{x}}(M)$ in the next section where we are only concerned with parameters.

\section{Content of local cohomology with respect to parameters}\label{par} 
\begin{prop}
\label{content-ineq} Let $x_1,\dots,x_d$ be elements of a
commutative ring $R$ and let $0\to L\to M\to N\to 0$ be a short
exact sequence of finite $R$-modules. Then
$$\fh^d_{\underline{x}}(M)\leq \fh^d_{\underline{x}}(L)+\fh^d_{\underline{x}}(N).$$
\end{prop}
\begin{proof}[Proof]
Let $I=(x_1,\dots,x_d)$ and denote $(x^t_1,\dots,x^t_d)$ by $I_t$
for all integers $t\geq 1$. It is clear that one has the following
short exact sequence
$$0\to L/L\cap I_tM\to M/I_tM\to N/I_tN\to 0.$$
Hence $\cL_I(M/I_tM)\leq \cL_I(L/L\cap I_tM)+\cL_I(N/I_tN)$ by
\cite[Proposition 1.1(d)]{quasi-length}. Since there is a natural surjection
$L/I_tL\twoheadrightarrow L/L\cap I_tM$, one has $\cL_I(L/L\cap I_tM)\leq
\cL_I(L/I_tL)$ by \cite[Proposition 1.1(d)]{quasi-length}. Therefore,
$$\cL_I(M/I_tM)\leq \cL_I(L/I_tL)+\cL_I(N/I_tN).$$
Dividing both sides by $t^d$ and taking limit over $t\to \infty$
give us the desired inequality.
\end{proof}

\begin{prop}
\label{content-free} Let $R$ be a $d$-dimensional equicharacteristic
local ring and let $x_1,\dots,x_d\in R$ be a system of parameters.
Then
$$\fh^d_{\underline{x}}(R^n)=n$$ for all integers $n\geq 1$.
\end{prop}
\begin{proof}[Proof]
It follows directly from \cite[Proposition 3.4]{quasi-length} that the
condition that $\fh^d_{\underline{x}}(R^n)=n$ is equational, and
hence reduction to characteristic $p>0$ is applicable. Therefore, it
suffices to prove $\fh^d_{\underline{x}}(R^n)=n$ in characteristic
$p>0$.\par

Use induction on $n$. When $n=1$, this is precisely \cite[Theorem
4.7]{quasi-length}.\par

Assume that $n\geq 2$. The short exact sequence $0\to R\to R^n\to
R^{n-1}\to 0$ and Proposition \ref{content-ineq} imply that
$\fh^d_{\underline{x}}(R^n)\leq \fh^d_{\underline{x}}(R)+\fh^d_{\underline{x}}(R^{n-1})$. By
induction, $\fh^d_{\underline{x}}(R^{n-1})=n-1$, and hence
$\fh^d_{\underline{x}}(R^n)\leq n$.\par

Assume that $\fh^d_{\underline{x}}(R^n)<n$. Let $I=(x_1,\dots,x_d)$
and denote $(x^t_1,\dots,x^t_d)$ by $I_t$ for all integers $t\geq
1$. By \cite[Theorem 2.4]{quasi-length} and \cite[Remark 3.7]{quasi-length}, there exist
an integer $q=p^e$ and a filtration
$$0\subsetneq M_1\subsetneq\cdots\subsetneq M_h=R^n/I_qR^n$$
of $R^n/I_qR^n$ with $M_{j+1}/M_j$ a homomorphic image of $R/I$ and
$h<nq^d$. As a consequence. one has
$$\lambda(R^n/I_qR^n)\leq \sum^{h-1}_{j=0}\lambda(M_{j+1}/M_j)\leq h\lambda(R/I).$$
 Applying the $e$-th Frobenius
functor to the above filtration, one has a filtration (with a slight
abuse of notation)
$$0\subsetneq M^{[q]}_1\subsetneq\cdots\subsetneq M^{[q]}_h=R^n/I_{q^2}R^n$$
of $R^n/I_{q^2}R^n$ with $M^{[q]}_{j+1}/M^{[q]}_j$ a homomorphic
image of $R/I_q$. Consequently, one has the following filtration of
$R^{n^2}/I_{q^2}R^{n^2}$
$$0\subsetneq F_1\subsetneq\cdots\subsetneq F_h=R^{n^2}/I_{q^2}R^{n^2}$$
where $F_j$ is the direct sum of $n$ copies of $M^{[q]}_j$. It is
clear that $F_{j+1}/F_j$ is a homomorphic image of $R^n/I_qR^n$ and
hence $$\lambda(R^{n^2}/I_{q^2}R^{n^2})\leq
\sum_{j=0}^{h-1}\lambda(F_{j+1}/F_j)\leq h\lambda(R^n/I_qR^n)\leq
h^2\lambda(R/I).$$ Similarly, one can prove that
$$n^l\lambda(R/I_{q^l})=\lambda(R^{n^l}/I_{q^l}R^{n^l})\leq h^l\lambda(R/I)$$
for all integers $l\geq 1$. Dividing both sides by $(nq^d)^l$, one
has
$$\frac{\lambda(R/I_{q^l})}{q^{ld}}\leq (\frac{h}{nq^d})^l\lambda(R/I)$$
for all integers $l\geq 1$. Since $h<nq^d$, it follows that
$$\lim_{l\to\infty}\frac{\lambda(R/I_{q^l})}{q^{ld}}=0$$
which is absurd because
$\lim_{l\to\infty}\frac{\lambda(R/I_{q^l})}{q^{ld}}$ is exactly
$e(I_q,R)$ (the multiplicity of $R$ with respect to $I_q$) which is
positive.\par

Therefore, one has $\fh^d_{\underline{x}}(R^n)=n$.
\end{proof}
\begin{cor}
Let $R$ be an equicharacteristic local ring and let
$x_1,\dots,x_s\in R$ be part of a system of parameters of $R$. Then
$$\fh^s_{\underline{x}}(R^n)=n$$ for all integers $n\geq 1$.
\end{cor}
\begin{proof}[Proof]
Use induction on $n$.\par When $n=1$, this follows from
\cite[Proposition 1.2(a)]{quasi-length} that $\fh^s_{\underline{x}}(R)\leq 1$.
Let $P$ be a prime ideal minimal over $I=(x_1,\dots,x_s)$ with
$\hgt(P)=s$. According to \cite[Proposition 2.5]{quasi-length}, we have
$\fh^s_{\underline{x}}(R)\geq \fh^s_{\underline{x}}(R_P)$. But
$\fh^s_{\underline{x}}(R_P)=1$ by Proposition \ref{content-free}
since $x_1,\dots,x_s$ is a system of parameters of $R_P$. Hence
$\fh^s_{\underline{x}}(R)\geq 1$. Therefore,
$$\fh^s_{\underline{x}}(R)=1.$$
Assume that $n\geq 2$ and $\fh^s_{\underline{x}}(R^{n-1})=n-1$. It
follows from Proposition \ref{content-ineq} that
$\fh^s_{\underline{x}}(R^n)\leq n$. According to \cite[Proposition
2.5]{quasi-length}, we have $\fh^s_{\underline{x}}(R^n)\geq
\fh^s_{\underline{x}}(R^n_P)$. But $\fh^s_{\underline{x}}(R^n_P)=n$
by Proposition \ref{content-free} since $x_1,\dots,x_s$ is a system
of parameters of $R_P$. Hence $\fh^s_{\underline{x}}(R^n)\geq n$.
Therefore,
$$\fh^s_{\underline{x}}(R^n)=n.$$
\end{proof}

\begin{cor}
\label{content-domain} Let $R$ be a $d$-dimensional
equicharacteristic local domain and let $x_1,\dots,x_d\in R$ be a
system of parameters. Then
\[\fh^d_{\underline{x}}(M)=\rk(M)\]
for all finitely generated $R$-modules $M$.\par In particular, under
the same hypotheses, $\fh^d_{\underline{x}}$ is additive, i.e.,
given a short exact sequence $0\to L\to M\to N\to 0$ of finitely generated 
$R$-modules, we have
$$\fh^d_{\underline{x}}(M)=\fh^d_{\underline{x}}(L)+\fh^d_{\underline{x}}(N).$$
\end{cor}
\begin{proof}[Proof]
Let $I=(x_1,\dots,x_d)$ and let $T$ denote that torsion submodule of
$M$. Then $\dim(T)<d$ and hence $H^d_I(T)=0$. By 
\cite[Proposition 2.2]{quasi-length}, one has $\fh^d_{\underline{x}}(T)=0$. The short exact
sequence $0\to T\to M\to M/T\to 0$ and Proposition
\ref{content-ineq} imply that
$$\fh^d_{\underline{x}}(M/T)\leq \fh^d_{\underline{x}}(M)\leq \fh^d_{\underline{x}}(T)+\fh^d_{\underline{x}}(M/T)=\fh^d_{\underline{x}}(M/T).$$
Hence, we may assume that $M$ is torsion-free. Consequently, there
is a short exact sequence $0\to M\to R^r\to R^r/M\to 0$ where
$r=\rk(M)$ and $\dim(R^r/M)<d$ (hence
$\fh^d_{\underline{x}}(R^r/M)=0$). Thus, one has
$$\rk(M)=r=\fh^d_{\underline{x}}(R^r)\leq
\fh^d_{\underline{x}}(M)+\fh^d_{\underline{x}}(R^r/M)=\fh^d_{\underline{x}}(M).$$
Let $g_1,\dots,g_r$ be elements in M whose images in
$(R\backslash{0})^{-1}M$ form a basis, and let
$M'=\sum^{r}_{j=1}Rg_j$. Then one has a short exact sequence $0\to
M'\to M\to M/M'\to 0$. Since $(R\backslash{0})^{-1}(M/M')=0$, it
follows that $\dim(M/M')<d$ and hence
$\fh^d_{\underline{x}}(M/M')=0$. By Proposition \ref{content-ineq},
one has
$$\fh^d_{\underline{x}}(M)\leq \fh^d_{\underline{x}}(M')+\fh^d_{\underline{x}}(M/M')=\fh^d_{\underline{x}}(M').$$
It is clear that there is a surjection $R^r\to M'$ and hence
$\fh^d_{\underline{x}}(M')\leq \fh^d_{\underline{x}}(R^r)=r$.
Therefore, $\fh^d_{\underline{x}}(M)\leq r$. This finishes the
proof.
\end{proof}

From Corollary \ref{content-domain}, one may expect quasi-length to
be additive under the same hypotheses as well. However, quasi-length
does not respect direct sum even when $R$ is a DVR, as shown in the
following example.

\begin{ex}
Let $R=k[[x]]$, let $I=(x^2)$, let $M=(x)/(x^4)$ and let
$N=(x)/(x^2)$. Then it is easy to see that $\mathcal{L}_I(N)=1$.
Since $M$ is not killed by $I$, one has $\mathcal{L}_I(M)\geq 2$; on
the other hand, one has an $I$-filtration
$$0\subsetneq M_1=R\cdot x^3\subsetneq M_2=R\cdot x^3+R\cdot x=M,$$
hence $\mathcal{L}_I(M)=2$.\par

We will prove that $\mathcal{L}_I(M\oplus N)=2\neq 3$ by
constructing a filtration of $M\oplus N$ with only 2 factors that
are homomorphic images of $R/I$ ({\itshape a priori},
$\mathcal{L}_I(M\oplus N)\geq
\max\{\mathcal{L}_I(M),\mathcal{L}_I(N)\}=2$).\par

Let $L_1=R\cdot (x^2,x)$ and $L_2=R\cdot (x^2,x)+R\cdot (x,x)$. It
is easy to check that
\begin{enumerate}
\item $IL_1=0$;
\item $IL_2\subset L_1$ and $L_2/L_1$ is cyclic;
\item $L_2=M\oplus N$.
\end{enumerate}
Therefore
$$0\subsetneq L_1\subsetneq L_2=M\oplus N$$
is the desired filtration.
\end{ex}

When $R$ is not a domain, we can't expect $\fh^d_{\underline{x}}$ to
be additive.

\begin{ex}
Let $R=k[[u,v]]/(uv)$ where $u$ and $v$ are indeterminates over
a field $k$. Set $x=u+v$, then $x$ is a system of parameter of $R$
($\dim(R)=1$). Since $R$ is reduced and has only two minimal prime
ideals $(u)$ and $(v)$, we have a short exact sequence
$$0\to R\to R/(u)\oplus R/(v)\to R/(u,v)\to 0.$$
It is clear that $\fh^1_{\underline{x}}(R/(u,v))=0$. As a
consequence of Proposition \ref{content-ineq}, we have
$\fh^1_{\underline{x}}(R/(u)\oplus R/(v))\leq
\fh^1_{\underline{x}}(R)=1$. But, it is easy to check that
$$\fh^1_{\underline{x}}(R/(u))=\fh^1_{\underline{x}}(R/(v))=1,$$
and hence $$1=\fh^1_{\underline{x}}(R/(u)\oplus R/(v))\neq
\fh^1_{\underline{x}}(R/(u))+\fh^1_{\underline{x}}(R/(v))=2.$$
\end{ex}
This example suggests that $\fh^d_{\underline{x}}(R)$ should not
exceed $\fh^d_{\underline{x}}(R/P)$ for all minimal prime ideals
$P$. The precise formula is given in the following theorem.

\begin{thm}
\label{main-thm} 
Let $R$ be a $d$-dimensional equicharacteristic
local ring and suppose that $x_1,\dots,x_d$ be a system of parameters of $R$.
Let $\Min_d(R)$ denote the set of minimal prime ideals $P$ of $R$ with
$\dim(R/P)=d$. Then for all finitely generated $R$-modules $M$,
$$\fh^d_{\underline{x}}(M)=\max\{\rk_{R/P}(M/PM)|P\in\Min_d(R)\}.$$
\end{thm}
\begin{proof}[Proof]
If $\dim(M)<d$, then $H^d_{(\underline{x})}(M)=0$ and hence
$\fh^d_{\underline{x}}(M)=0$. It is clear that $\rk_{R/P}(M/PM)=0$
for all $P\in \Min(R)$. Thus, the theorem is true when
$\dim(M)<d$.\par

Assume that $\dim(M)=d$. Let $N$ be the largest submodule of $M$
with dimension smaller than $d$. The short exact sequence $0\to N\to
M\to M/N\to 0$ and Proposition \ref{content-ineq} imply that
$\fh^d_{\underline{x}}(M)=\fh^d_{\underline{x}}(M/N)$. Hence, we may
assume that $M$ is of pure dimension $d$. Since $M$ is naturally an
$R/\Ann_R(M)$-module, replacing $R$ by $R/\Ann_R(M)$, we may assume
that $M$ is faithful (consequently $R$ is also of pure dimension
$d$). Let $P$ be a minimal prime ideal of $R$ and let $y_i$ denote
the image of $x_i$ in $R/P$. Then by \cite[Proposition 2.5]{quasi-length},
$\fh^d_{\underline{x}}(M)\geq \fh^d_{\underline{y}}(M/PM)$; the
latter is $\rk_{R/P}(M/PM)$ by Corollary \ref{content-domain}.
Therefore,
$$\fh^d_{\underline{x}}(M)\geq\max\{\rk_{R/P}(M/PM)|P\in\Min_d(R)\}.$$\par

Without loss of generality, we may assume that $\Min_d(R)\subset
\Supp(M)$. For each $P\in \Min_d(R)$, we have a filtration
$$0\subsetneq L(P)_1\subsetneq \cdots\subsetneq L(P)_{\rk_{R/P}(M/PM)}=M_P$$
with $L(P)_{j+1}/L(P)_j\cong
R_P/PR_P(=((R/P)\backslash\{0\})^{-1}R/P)$. Set
$$h=\max\{\rk_{R/P}(M/PM)|P\in\Min_d(R)\}$$ and extend the above
filtration trivially (if necessary) as follows
$$0\subsetneq L(P)_1\subsetneq \cdots\subsetneq L(P)_{\rk_{R/P}(M/PM)}=\cdots=L(P)_h=M_P.$$
Let $L_j=\oplus_{P\in\Min_d(R)}L(P)_j$, then we have a filtration
\begin{equation}
\label{local-filtration} 0\subsetneq L_1\subsetneq\cdots\subsetneq
L_h=\oplus_{P\in\Min_d(R)}M_P
\end{equation}
with $L_{j+1}/L_j$ a homomorphic image of
$\oplus_{P\in\Min_d(R)}R_P$.\par

Let $W=R\backslash\cup_{P\in\Min_d(R)}P$. It is clear that
$W^{-1}R=\oplus_{P\in\Min_d(R)}R_P$ and
$W^{-1}M=\oplus_{P\in\Min_d(R)}M_P$ (since $W^{-1}R$ is artinain).
Thus, the filtration (\ref{local-filtration}) gives us a filtration
of $W^{-1}M$ with each factor a homomorphic image of $W^{-1}R$,
which in turn gives us a filtration of $M$
$$0\subsetneq M_1\subsetneq \cdots\subsetneq M_h=M$$
such that there is a surjection $W^{-1}R\to W^{-1}(M_{j+1}/M_j)$ for
each $j$.\par

Let $g$ be an element of $M_{j+1}/M_j$ whose image in
$W^{-1}(M_{j+1}/M_j)$ generates $W^{-1}(M_{j+1}/M_j)$. Then we have
a short exact sequence $0\to Rg\to M_{j+1}/M_j\to
(M_{j+1}/M_j)/Rg\to 0$. Since $W^{-1}Rg=W^{-1}(M_{j+1}/M_j)$, we
have $$W^{-1}((M_{j+1}/M_j)/Rg)=0.$$ Thus,
$\dim((M_{j+1}/M_j)/Rg)<d$ and hence
$\fh^d_{\underline{x}}((M_{j+1}/M_j)/Rg)=0$. It follows from
Proposition \ref{content-ineq} that
$$\fh^d_{\underline{x}}(M_{j+1}/M_j)\leq \fh^d_{\underline{x}}(Rg)+\fh^d_{\underline{x}}((M_{j+1}/M_j)/Rg)=\fh^d_{\underline{x}}(Rg)\leq 1.$$
Furthermore, as a consequence of Proposition \ref{content-ineq}, we
have
$$\fh^d_{\underline{x}}(M)\leq \sum^{h-1}_{j=1}\fh^d_{\underline{x}}(M_{j+1}/M_j)\leq h=\max\{\rk_{R/P}(M/PM)|P\in\Min_d(R)\}$$
which finishes the proof.
\end{proof}

Recall that, if $M$ is a finitely generated module over a local ring
$R$, then a {\it system of parameters} of $M$ is a sequence of $d=\dim(M)$
elements $x_1,\dots,x_d$ in $R$ such that
$\dim(M/(x_1,\dots,x_d)M)=0$.\par

\begin{cor}
Let $R$ be an equicharacteristic  local ring and let $M$ be a finite
$R$-module of dimension $d$. Let $\Min_d(M)$ denote the set of
prime ideals $P$ minimal over $\Ann_R(M)$ with $\dim(R/P)=\dim(M)$.
Then for any system of parameters $x_1,\dots,x_d$ of $M$,
$$\fh^d_{\underline{x}}(M)=\max\{\rk_{R/P}(M/PM)|P\in\Min(M)\}.$$
\end{cor}
\begin{proof}[Proof]
It is clear that $M$ is naturally an $R/\Ann_R(M)$-module. By the
remark on page 11 in \cite{quasi-length}, we may replace $R$ by $R/\Ann_R(M)$
and assume that $x_1,\dots,x_d$ is also a system of parameters of
$R$. Then Theorem \ref{main-thm} finishes the proof.
\end{proof}
\section{A non-zero local cohomology module whose content is 0}\label{notrob}   
Let $R=A[x,y,u,v]$ where $A$ is any Noetherian commutative ring. Let $P=(x,y)$, $Q=(u,v)$, and $I=(x_1,x_2,x_3)$ where $x_1=xu$, $x_2=yv$, and $x_3=xv-yu$. Then one can check that $J=P\cap Q$ is the radical of $I$. It is left open in \cite[Example 3.14]{quasi-length} whether $x_1,x_2,x_3$ forms a $\Q$-sequence when $A = k$ is a field.  
Our goal of this section is to prove that
\[\underline{\fh}^3_{\underline{x}}(R)=\fh^3_{\underline{x}}(R)=0\] 
no matter what $A$ is, and consequently that $\underline{x}$ is not a $\Q$-sequence for any choice of $A$. 
To the best of our knowledge, this is the 
first example of a nonzero local cohomology module whose content is 0.\par

We begin with an easy comparison between the ideals $I_t=(x^t_1,x^t_2,x^t_3)$ and $I'_t=(x^t_1,x^t_2,x^tv^t+y^tu^t)$.
\begin{prop}
\label{comparison}
We have
\begin{enumerate}
\item $(xv)^2,(yu)^2\in I$
\item $I_{4t}\subseteq I'_t$
\item $I'_{12t}\subseteq I_t$
\end{enumerate}
\end{prop}
\begin{proof}[Proof]
(1). It is clear that we have
\[(xv)^2=x^2_3-yux_3-x_1x_2\in I\ \text{and}\ (yu)^2=x^2_3-xvx_3-x_1x_2\in I.\]
(2). It suffices to show that $x^{4t}_3\in I'_t$. It is wasy to see that $(xv)^{2t},(yu)^{2t}\in I'_t$. Since $x^{4t}_3$ is a linear combination of $(xv)^i(yu)^{4t-i}$ in which at least wither $i\geq 2t$ or $4t-i\geq 2t$, we have that $x^{4t}_3\in I'_t$. \par
(3). It suffices to show that $(xv)^{12t}+(yu)^{12t}\in I_t$. First we note that $I^{3t}\subseteq I_t$ since $I$ is generated by 3 elements. Consequently,
\[(xv)^{6t}=((xv)^2)^{3t}\in I^{3t}\subseteq I_t.\]
And similarly, $(yu)^{6t}\in I_t$. Now we have
\[(xv)^{12t}+(yu)^{12t}=(xv+yu)^{12t}-(xv)^{6t}\alpha-(yu)^{6t}\beta,\]
for some $\alpha,\beta\in R$. Hence, $(xv)^{12t}+(yu)^{12t}\in I_t$. 
\end{proof}

Let $T$ be a commutative Noetherian ring and let $\fa=(a_1,\dots,a_t)$ and $\fb=(b_1,\dots,b_t)$ be ideals of $R$. Let $H^i(\underline{a})$ (respectively, $H^i(\underline{b})$) denote the $(t-i)$-th homology of the Koszul complex associated to $a_1,\dots,a_t$ (respectively, to $b_1,\dots,b_t$). We recall the following result from \cite{StruckradVogelBuchsbaumRings}.

\begin{lem}[Lemma 1.5 in \cite{StruckradVogelBuchsbaumRings}]
Let $T,\fa,\fb,H^i(\underline{a}),H^i(\underline{b})$ as above. If $\fa\subseteq \fb$, then we have $t+1$ commutative diagrams:
\[\xymatrix{
H^i(\underline{b})\ar[r]\ar[d] &H^i(\underline{a})\ar[d]\\
H^i_{\fb}(T)\ar[r] &H^i_{\fa}(T)
}\]
\end{lem}

As a consequence of this lemma, one can see that the map $R/\fa\to H^t_{\fa}(R)$ is canonical in the sense that, if $(a_1,\dots,a_t)=(\alpha_{ij})(b_1,\dots,b_t)$ for a matrix $(\alpha_{ij})$ with $\alpha_{ij}\in T$ (i.e. $\fa\subseteq \fb$) and $\sqrt{\fa}=\sqrt{\fb}$, then we have a commutative diagram
\begin{equation}
\label{canonicity}
\xymatrix{
R/\fb \ar[rr]^{\text{det}(\alpha_{ij})}\ar[dr] & & R/\fa \ar[dl]\\
 & H^t_{\fa}(R) & 
}
 \end{equation}
\begin{thm}
Let $R,I,J,P,Q,\underline{x}$ as above. Then $\fh^3_{\underline{x}}(R)=\underline{\fh}^3_{\underline{x}}(R)=0$.
\end{thm}
\begin{proof}[Proof] 
It is clear that $ \varinjlim_tR/I_t=H^3_I(R)=H^3_J(R)$, where the last equality holds since $J$ is the radical of $I$. We also note that, by the Mayer-Vietoris sequence
\[\cdots\to 0=H^3_P(R)\oplus H^3_Q(R)\to H^3_J(R)\to H^4_{P+Q}(R)\to H^4_P(R)\oplus H^4_Q(R)=0\to\cdots,\]
we have $H^3_J(R)=H^4_{(x,y,u,v)}(R)$. Let $\fm=(x,y,u,v)$ and $\fm_t=(x^t,y^t,u^t,v^t)$ for $t\in \mathbb{N}$. \par

Since every element of $H^4_{\fm}(R)$ is killed by $\fm_t$ for some suitable $t$, we know that the image of $R/(xu,yv,xv+yu)$ in $H^3_I(R)=H^3_j(R)=H^4_{\fm}(R)$ is killed by $\fm_c$ for some $c\in \mathbb{N}$.\par
Let $F_t:R\to R$ be the endomorphism that sends $x$ ($y,u,v$, respectively) to $x^t$ ($y^t,u^t,v^t$, respectively) and is the  identity on $A$. Applying $F_t$ to the map 
$$
R/(xu,yv,xv+yu)\to H^4_{\fm}(R),
$$
we see that the image of $R/I'_t$ in $H^4_{\fm_t}(R)=H^4_{\fm}(R)$ is killed by $\fm_{ct}$.\par
Since we always have a surjection $R/P_{ct}\to R/(P_{ct}+Q_{ct})=R/\fm_{ct}$, we have that 
\[\cL_J(R/\fm_{ct})\leq \cL_J(R/P_{ct})).\]
Moreover we have that 
\[\cL_J(R/P_{ct})\leq \cL_P(R/P_{ct})=c^2t^2\]
since $J\subset P$. Therefore, we have that 
\[\cL_J(R/\fm_{ct})\leq c^2t^2,\]
and hence
\[\cL_J(\im(R/I'_t\to H^4_{\fm}(R)))\leq c^2t^2.\]
Since $I'_{12t}\subseteq I_t$, there is a surjection $\im(R/I'_{12t}\to H^4_{\fm}(R))\to \im(R/I_t\to H^4_{\fm}(R))$. Hence
\begin{align}\cL_I(\im(R/I_t\to H^4_{\fm}(R)))&\leq \cL_J(\im(R/I_t\to H^4_{\fm}(R)))\notag\\
&\leq \cL_J(\im(R/I'_{12t}\to H^4_{\fm}(R)))\notag\\
&\leq c^2(12t)^2=144c^2t^2,\notag
\end{align}
where the first inequality holds since $I\subset J$.\par
As discussed in \S\ref{prel}, $\im(R/I_t\to H^3_I(R)=H^4_{\fm}(R))$ is isomorphic to $R/(I_t)^{\lm}$. Therefore,
\[\underline{\fh}^3_{\underline{x}}(R)\leq \lim_{t\to \infty}\frac{144c^2t^2}{t^3}=0.\]
To finish the proof of our theorem, we will show that $\fh^3_{\underline{x}}(R)$ is also 0\footnote{Note that $\underline{\fh}^3_{\underline{x}}(R)=0$ already implies that $x_1,x_2,x_3$ does not form a $\Q$-sequence in $R$ by \cite[Theorem 3.8]{quasi-length}.}, i.e., we will prove that
\[\lim_{t\to\infty}\inf\{\frac{\cL_I(R/I_t)}{t^3}\}=0.\] 
For any $s\geq t$, we have that
\begin{align} &\ ((xu)^{t+s},(yv)^{t+s},x^{t+s}v^{t+s}+y^{t+s}u^{t+s})\notag\\
&=((xu)^t,(yv)^t,x^tv^t+y^tu^t)\begin{pmatrix}(xu)^s&0&-y^su^{s-t}v^t\\0&(yv)^s&-x^su^tv^{s-t}\\0&0&x^{s}v^{s}+y^{s}u^{s}\end{pmatrix}\notag
\end{align}
Hence by the commutative diagram (\ref{canonicity}), we have that $\im(R/I'_t\to H^3_{I}(R))$ is $H_t$, the cyclic module generated by the image of $(xu)^s(yv)^s(x^sv^s+y^su^s)$ in $H^3_{I}(R)$ through $R/I'_{s+t}\to H^3_I(R)$. For $s\gg 0$, 
one has that $H_t$ is isomorphic to $H'_{s+t}$, the submodule of $R/I'_{s+t}$ generated by $(xu)^s(yv)^s(x^sv^s+y^su^s)$. Therefore, we have 
\[\cL_J(H'_{s+t})=\cL_J(H_t)=\cL_J(\im(R/I'_t\to H^3_{I}(R)))\leq c^2t^2.\]
Next we want to get a upper bound for the quasilength of 
\[R/(I'_{s+t},(xu)^s(yv)^s(x^sv^s+y^su^s))\cong \frac{R/I'_{s+t}}{H'_{s+t}}.\]
More precisely, we wish to show that 
\[\cL_I(R/(I'_{s+t},(xu)^s(yv)^s(x^sv^s+y^su^s)))\leq (t+s)^3-t^3.\]
To this end we will construct an $I$-filtration of $R/(I'_{s+t},(xu)^s(yv)^s(x^sv^s+y^su^s))$ with $(s+t)^3-t^3$ factors. 
First we linearly order the elements in
\[\cS:=\{(xu)^{e_1}(yv)^{e_2}(x^{e_3}v^{e_3}+y^{e_3}u^{e_3})|0\leq e_1,e_2,e_3\leq t+s-1\}\]
as follows: we set $$
(xu)^{e_1}(yv)^{e_2}(x^{e_3}v^{e_3}+y^{e_3}u^{e_3})\succeq (xu)^{e'_1}(yv)^{e'_2}(x^{e'_3}v^{e'_3}+y^{e'_3}u^{e'_3}),$$ 
if $e_1+e_2>e'_1+e'_2$, or $e_1+e_2=e'_1+e'_2$ and $e_1>e'_1$, or $e_1=e'_1$ and $e_2=e'_2$ and $e_3\geq e'_3$.  
Now we remove all elements $(xu)^{e_1}(yv)^{e_2}(x^{e_3}v^{e_3}+y^{e_3}u^{e_3})$ with $e_1,e_2,e_3\geq s$ from $\cS$. Hence according to our order, the remaining $(t+s)^3-t^3$ elements can be listed as
\begin{equation}
\label{remaining-elements}
(xu)^{t+s-1}(yv)^{t+s-1}(xv+yu)^{s-1}, (xu)^{t+s-1}(yv)^{t+s-1}(xv+yu)^{s-2},\cdots
\end{equation}
Let $M_h$ be the submodule of $R/(I'_{s+t},(xu)^s(yv)^s(x^sv^s+y^su^s))$ generated by the first $h$ elements 
in (\ref{remaining-elements}). To see that the $M_h$ give an $I$-filtration of 
$$
R/(I'_{s+t},(xu)^s(yv)^s(x^sv^s+y^su^s)),$$ 
it suffices to see that $IM_{h}\subseteq M_{h-1}$. It is clear that $(xu,yv)M_{h}\subseteq M_{h-1}$. It remains to prove that $(xv+yu)M_h\subseteq M_{h-1}$. Let \[g(e_1,e_2,e_3)=(xu)^{e_1}(yv)^{e_2}(x^{e_3}v^{e_3}+y^{e_3}u^{e_3}).\] Assume that $g(e_1,e_2,e_3)\in M_h\backslash M_{h-1}$. If $e_3=0$, then it is clear that $$
(xv+yu)g(e_1,e_2,0)=g(e_1,e_2,1)\in M_{h-1}.
$$ If $0<e_3<s-1$, then 
\begin{align}
&\ (xv+yu)g(e_1,e_2,e_3)\notag\\
 &=(xu)^{e_1}(yv)^{e_2}(x^{e_3+1}v^{e_3+1}+y^{e_3+1}u^{e_3+1})+(xu)^{e_1}(yv)^{e_2}(xvy^{e_3}u^{e_3}+yux^{e_3}v^{e_3})\notag\\
&=g(e_1,e_2,e_3+1)+(xu)^{e_1+1}(yv)^{e_2+1}(y^{e_3-1}u^{e_3-1}+x^{e_3-1}v^{e_3-1})\in M_{h-1}.\notag
\end{align}
When $e_3=s-1$, a similar calculation shows that
\begin{align}
(xv+yu)g(e_1,e_2,s-1)&=g(e_1,e_2,s)+(xu)^{e_1+1}(yv)^{e_2+1}(y^{s-2}u^{s-2}+x^{s-2}v^{s-2})\notag\\
&=g(e_1,e_2,s)+g(e_1+1,e_2+1,s-2)\notag
\end{align} 
When $e_1,e_2\geq s$, we know that $g(e_1,e_2,s)=0$ in $R/(I'_{s+t},(xu)^s(yv)^s(x^sv^s+y^su^s))$; when one of $e_1$ and $e_2$ is less than $s$, we have that $g(e_1,e_2,s)\in M_{h-1}$. And it is clear that $g(e_1+1,e_2+1,s-2)\in M_{h-1}$. Hence $(xv+yu)g(e_1,e_2,s-1)\in M_{h-1}$. This finishes the proof that $\cL_I(R/(I'_{s+t},(xu)^s(yv)^s(x^sv^s+y^su^s)))\leq (t+s)^3-t^3$.\par
By the short exact sequence
\[0\to H'_{s+t}\to R/I'_{s+t}\to R/(I'_{s+t},(xu)^s(yv)^s(x^sv^s+y^su^s))\to 0,\]
we have that 
\[\cL_I(R/I'_{s+t})\leq \cL_I(H'_{s+t})+\cL_I(R/(I'_{s+t},(xu)^s(yv)^s(x^sv^s+y^su^s)))\leq c^2t^2+(s+t)^3-t^3.\]
Therefore, for $n\gg 0$, we have $\cL_I(R/I'_n)<n^3$. We may choose such an integer $n$ that is divisible by 12, i.e., $n=12n_0$. We can write $\cL_I(R/I'_n)=\epsilon n^3$ with $0\leq \epsilon<1$. After applying $F_n$ repeatedly, we have that $\cL_I(R/I'_{n^l})\leq \epsilon^ln^{3l}$. Now according to Proposition \ref{comparison}(3), we have
\[\cL_I(R/I_{12^{l-1}n^l_0})\leq \cL_I(R/I'_{12^ln^l_0})=\cL_I(R/I'_{n^l})\leq \epsilon^ln^{3l}\]
and hence
\[\fh^3_{\underline{x}}(R)\leq \frac{\cL_I(R/I_{12^{l-1}n^l_0})}{(12^{l-1}n^l_0)^3}\leq \frac{\epsilon^ln^{3l}}{(12^{l-1}n^l_0)^3}=12^3\epsilon^l\]
for all $l\geq 1$. Therefore,
\[\fh^3_{\underline{x}}(R)=0.\]
\end{proof}

\section{Robust Closure}\label{robcl}   
\begin{defn}\label{defforce} We beginning by recalling the definitions of forcing algebras and generic forcing algebras for a triple ($M,N,u$) (\cite[4.3]{solid-closure}) where $N\subseteq M$ are finitely generated modules over a Noetherian ring $R$ and $u$ is an element of $M$. An $R$-algebra $S$ is called a {\it forcing algebra} for the triple $(M,N,u)$ if the image $1\otimes u$ of $u$ in $S\otimes_RM$ is in $\im(S\otimes_RN\to S\otimes_RM)$. Let $R^{\beta}\xrightarrow{A}R^{\alpha}\to M/N\to 0$ be a finite presentation of $M/N$, where $A=(a_{ij})$ is an $\alpha\times \beta$ matrix over $R$. Let $r=(r_1,\dots,r_{\alpha})\in R^{\alpha}$ represent $\bar{u}\in R^{\alpha}/\im(A)\cong M/N$. Let $Z_1,\dots,Z_{\beta}$ be indeterminates over $R$. Then 
\[\frac{R[Z_1,\dots,Z_{\beta}]}{(r_i-\sum^{\beta}_{j=1}a_{ij}Z_j: 1 \leq i \leq \alpha).}\]
is called a {\it  generic forcing algebra} for the triple $(M,N,u)$ for the data $A,r$. It is clear that the image $1\otimes u$ of $u$ in $S\otimes_RM$ is contained in $\im(S\otimes_RN\to S\otimes_RM)$. Given any other forcing algebra $S'$ for the triple $(M,N,u)$, there is a ring homomorphism $S\to S'$; this justifies the word `generic.'\end{defn}
\begin{rem}\label{remforce}  It is important to note that a forcing algebra (respectively, a generic forcing algebra) for $(M,N,u)$ is the same
as a forcing algebra (respectively, a generic forcing algebra) for $(M/N, 0, \overline{u})$, 
 where $\overline{u}$ is the image of $u$ in  $M/N$.
 
 Note also that if $S$ is a forcing algebra for $(M,N,u)$ and $S \to T$ is an $R$-algebra homomorphism, then $T$ is also a forcing
 algebra for $(M,N,u)$.
 
 Furthermore, note that if $N \inc N' \inc M \inc M'$ are finitely generated $R$-modules and $u \in M$,  then a forcing algebra 
for $(M, N, u)$ is a forcing algebra for $(M',N',u)$,  since we have an $R$-linear map $M/N \to M'/N'$
that sends the image of $u$ in the first module to the image of $u$ in the second module. The case where $M = M'$ and the
case where $N = N'$ are of particular interest.
 \end{rem}

The following proposition (\cite[Proposition 4.6]{solid-closure}) will be useful to us.
\begin{prop}
\label{two-generic-forcing-algebras}
If $R$ is a Noetherian ring, $N\subseteq M$ are finitely generated $R$-modules, $u\in M$, and $S,T$ are two generic forcing algebras for $(M,N,u)$ for possibly different data, then there are finite sets of indeterminates \textbf{Y, \textbf{Z}} such that $S[\textbf{Y}]=T[\textbf{Z}]$ as $R$-algebras.
\end{prop}

For any Noetherian ring $R$, following \cite[3.1]{solid-closure}, we refer to the $R$-algebra obtained by completing $R$ localized at a maximal ideal and then killing a minimal prime as a {\it complete local domain} of $R$. Note that if $R$ is local and $\fp$ is a minimal prime of $R$, then
every minimal prime of $\fp \widehat{R}$ is a minimal prime of  $(0)$ in $\widehat{R}$ lying over $\fp$.  (If $a \notin \fp$ it is a nonzerodivisor
on $R/\fp$ and, hence, of the flat $(R/\fp)$-algebra $\widehat{R}/\fp \widehat{R}$.  Thus $\fq$ lies over $\fp$.  Moreover, we can
choose $a \in R-\fp$ such that  $a\fp^n = 0$ for some integer $n \geq 1$.  When we localize at $\fq$,  $\fp$ therefore becomes
nilpotent, and $\fq$ becomes nilpotent modulo $\fp \widehat{R}$.)  It follows that the set of complete local domains $D$ of the local
ring $R$ is the same as the set of complete local domains of the $R/\fp$ for $\fp$ a minimal prime of $R$:  the map  $R \to \widehat{R} \to D$ factors uniquely  $R \to \widehat{R} \to \widehat{R}/\fp R \to \widehat{R}/\fq = D$ where $\fq$ is a minimal prime both of $\fp\widehat{R}$ and of $(0)$ in 
$\widehat{R}$.

Next we collect the definition and some basic facts of solid closure.
\begin{defn}[Definition 1.1 in \cite{solid-closure}]
If $R$ is a domain, we shall say that an $R$-module $M$ is {\it solid} if $\Hom_R(M,R)\neq 0$. 
We shall say that an $R$-algebra $S$ is {\it solid} if it is solid as an $R$-module. 
\end{defn}
\begin{rem}\label{solidrem} $S$ is solid as an $R$-algebra if
and only if there is an $R$-module map $\theta: S \to R$ such that $\theta(1) \not= 0$.  (If $\theta(s_0) \not=0$,
one may replace $\theta$ by $\theta'$ where $\theta'(s) := \theta(s_0s)$.)
A module-finite extension $S$ of a Noetherian domain $R$ is always a solid $R$-algebra (\cite[Proposition 2.1(i)]{solid-closure}). If $(R,\fm)$ is a complete local domain of Krull dimension $d$, then an $R$-module is solid if and only 
if $H^d_{\fm}(M)\neq 0$ (\cite[Corollary 2.4]{solid-closure}).\end{rem}
\begin{defn}[Definition 5.1 in \cite{solid-closure}]\label{defsolcl}
Let $R$ be a Noetherian ring, let $N\subseteq M$ be finitely generated $R$-modules, and let $u\in M$. If $R$ is a complete local domain we say that $u$ is in the {\it solid closure} $N^{\bigstar}$ of $N$ in $M$ over $R$ if there is a solid $R$-algebra $S$ such that the image $1\otimes u$ of $u$ in $S\otimes_RM$ is in $\im(S\otimes_RN\to S\otimes_RM)$. In other words, $u\in N^{\bigstar}$ if and only if $(M,N,u)$ has a forcing algebra $S$ that is solid as an $R$-algebra.  If there is as a solid forcing algebra, then the generic forcing algebra is solid.\par
In the general case, we say that $x$ is in the solid closure $N^{\bigstar}$ of $N$ in $M$ over $R$ is for every comlete local domain $D$ of $R$, the image of $x$ in $D\otimes_RM$ is in the solid closure of $\im(D\otimes_RN\to B\otimes_RM)$ in $D\otimes_RM$ over $B$. In other words, every complete local domain $D$ of $R$ has a solid $D$-algebra $S$ such that the image of $x$ in $S\otimes_Rm$ is in $\im(S\otimes_RN\to S\otimes_RM)$.  
\end{defn}
The following facts about solid closure will be used in the sequel:
\begin{enumerate}
\item Let $R$ be a Noetherian ring and $I$ an ideal of $R$. Then $I^{\bigstar}\subseteq \bar{I}$. Moreover if $I$ is principal then $I^{\bigstar}=\bar{I}$ (\cite[Theorem 5.10]{{solid-closure}}).
\item If $R$ is a Noetherian ring of positive characteristic and admits a completely stable test element and $N\subseteq M$ are finitely generated $R$-modules, then $N^*_M=N^{\bigstar}_M$. In particular, this is true in positive characteristic if $R$ is reduced and essentially of finite type over an 
excellent semilocal ring, and, in particular, it is true in positive characteristic for every complete local domain (\cite[Theorem 8.6]{{solid-closure}}).
\end{enumerate}

Recall the following definition from \cite{BrennerParasolidClosure}.
\begin{defn}[\cite{BrennerParasolidClosure}]
Let $(A,\fm)$ be a $d$-dimensional Noetherian commutative local ring. A {\it paraclass} in $H^d_{\fm}(A)$ is the cohomology class $[\frac{1}{x_1\dots x_d}]\in H^d_{\fm}(A)$ where $x_1,\dots,x_d$ is a system of parameters of $A$ and we think of $H^d_{\fm}(A)$ as the quotient of $A_{x_1 \cdots x_d}$
by $\sum_{i=1}^d A_{f_i}$ where $f_i = \prod_{j\not=i}x_j$.  

Let $R$ be an $A$-algebra. $R$ is called a {\it parasolid} $A$-algebra if the image of each non-zero paraclass $c\in H^d_{\fm}(A)$ in $H^d_{\fm R}(R)$ does not vanish.

More generally, let $A$ be a Noetherian commutative ring and let $R$ be an $A$-algebra. $R$ is called a {\it parasolid} 
$A$-algebra if $R_{\fm}$ is parasolid over $A_{\fm}$ for each maximal ideal $\fm$ of $A$.
\end{defn}
\begin{rem}\label{splitpar} It is clear that a parasolid algebra over a complete local domain is always solid from the local cohomology 
characterization of being solid given in Remark~\ref{solidrem}.  Note also that if $A \to R$ is split or even pure as a
map of $A$-modules, then $R$ is parasolid over $A$:  the condition of being split or pure is preserved by base
change to a complete local domain of $A$.  But when  $(A,\fm)$ is a complete local domain the map of 
local cohomology $H^d_{\fm}(A) \to H^d_{\fm}(R)$ is injective.
\end{rem}
\begin{defn}
\label{defn-robust-algebra}
Let $R$ be a complete local domain. An $R$-algebra $S$ is called {\it robust} if the image of every system of parameters for $R$ in $S$ forms a $\Q$-sequence in $S$.\par
In general, for any Noetherian ring $R$, an $R$-algebra $S$ is called {\it robust} if 
$D\otimes_RA$ is a robust $D$-algebra for every complete local domain $D$ of $R$. 
\end{defn}

Note that a robust algebra $S$ over a complete local domain $(R,\,\fm)$ of dimension $d$ is solid, since even the fact that one system of parameters
$\ux = x_1,\, \ldots, \, x_d$ is a Q-sequence on $S$ implies that $H_{(\ux)}^d(S) = H^d_{\fm}(S) \not= 0$.
We make the following conjecture:
\begin{conj}\label{robcj}  Every system of parameters for every local ring is a Q-sequence.  Hence, every Noetherian ring
is robust as an algebra over itself. \end{conj}

The issue  for Noetherian rings of Krull dimension at most $d$ reduces to the case 
of complete local domains of Krull dimension at most $d$, where it is equivalent to ask whether every
system of parameters is a Q-sequence. One may pass to the normalization of the complete local domain. 
The results of Hochster and Huneke \cite{quasi-length} show that this conjecture is true in equal characteristic 
(\cite[Theorem 4.1]{quasi-length})  and in dimension  at most 2 
(the normalized complete local domain is Cohen-Macaulay and one may apply
\cite[Proposition 1.2(c)]{quasi-length}.  Hence, by the results of Hochster and Huneke we have: 
\begin{thm}\label{equirob}  Let $R$ be a Noetherian ring such that $R\red$ contains a field or such that $R$ is of Krull dimension
at most 2.   Then  $R$ is robust. \end{thm}

The difficult result on the vanishing of content in \S\ref{notrob} enables us to show that parasolid algebras are not, 
in general, robust.  In the example given in the result below, $A \to R$ is actually split as a map of $A$-modules.
\begin{thm}
Let $R=\mathbb{C}[[x,y,u,v]]$ and $A=\mathbb{C}[[xu,yv,xv+yu]]$. Then $R$ is a parasolid $A$-algebra, in fact
$A \to R$ splits as a map of $A$-modules, but $R$ is not a robust $A$-algebra.
\end{thm}
\begin{proof}
Let $B=\mathbb{C}[[xu,xv,yu,yv]]$. Then it is easy to see that $B$ is Cohen-Macaulay and $xu,yv,xv+yu$ is a system of parameters of $B$. Therefore $B$ is a free $A$-module and hence the natural inclusion $A\hookrightarrow B$ splits (over $A$). On the other hand, $B$ is the invariant subring of $R$ under the $\mathbb{C}^*$-action given by:
\[c\circ x=cx,\ c\circ y=cy,\ c\circ u=\frac{1}{c}u,\ c\circ v=\frac{1}{c}v\]
for each $c\in \mathbb{C}^*$. And hence $B\hookrightarrow R$ splits (over $B$). Therefore $A\hookrightarrow R$ splits (over $A$), and consequently $R$ is a parasolid $A$-algebra by Remark~\ref{splitpar}.

However, from \S\ref{notrob} we know that $R$ is not a robust $A$-algebra.
\end{proof}
\begin{defn}
\label{defn-robust-closure}
Let $(R,\fm)$ be a complete local domain and $N\subseteq M$ finitely generated $R$-modules. Let $N\roi$ (or
${N\roi}_M$ if $M$ is not clear from context) denote the
submodule of $M$ generated by $N$ and all elements $u\in M$ such that a generic forcing 
$S$ for the triple $(M,N,u)$ is a robust $R$-algebra. (It is equivalent to say that there exists a forcing algebra for
$(M,N,u)$ that is robust.) We refer to $N\roi$ as the {\it immediate robust closure}
of $N$ in $M$.  Let $N\rob j$ be defined recursively by $N\rob 0 = N$ and $N\rob {j+1} = (N{\rob j})\roi$.  
The {\it robust closure} of $N$ in $M$ is defined to be $\bigcup_j N\rob j$, denoted by $N\ro$ or $N\ro_M$. 
Note that since the sequence $N \rob j$ is clearly nondecreasing as $j$ increases and $M$ is Noetherian, we 
have that $N\ro = N \rob j$ for all $j \gg 0$.  
\par
In general, for any Noetherian ring $R$ and finitely generated $R$-modules $N\subseteq M$, we say an element $u\in M$ is in the {\it robust closure} of $N$ in $M$ if $1\otimes u$ is contained in $\im(R'\otimes N\to D\otimes_R M)\ro_{D\otimes_R M}$ for every complete local domain $D$ of $R$.   
\end{defn}
\begin{rem}
In Definition~\ref{defn-robust-closure}, if we use a generic forcing algebra $T$ for $(M,N,u)$ for possibly different data, then 
whether the generic forcing algebra obtained is robust is not affected. To see this, note that by Proposition \ref{two-generic-forcing-algebras}, it suffices to prove that a system of parameters for $R$ forms a $\Q$-sequence in a generic forcing algebra $S$ for 
$(M,N,u)$ if and only if it forms a $\Q$-sequence in the polynomial 
ring $S[{\bf Y}]$ over $S$.  This is clear because we have $S$-algebra homomorphisms $S \to S[{\bf Y}] \to S$
whose composition is the identity, and these are also $R$-algebra homomorphisms. \par
Moreover, in defining $N\roi$ (and, hence, in defining all the $N\rob j$), we may take $S$ to be any forcing algebra
rather than a generic forcing algebra.  For then we may choose a generic forcing algebra $S_0$  such that
$R \to S$ factors $R \to S_0 \to S$,  and the fact that $S$ is robust implies that $S_0$ is robust.
\end{rem}

There are technical difficulties in working with robust closure. One is that we do not know that Noetherian rings are robust as algebras
over themselves, which was discussed above. Her are two others.  We do not know whether the tensor product of two
robust algebras over a complete local domain $R$ is robust, nor whether, if $S$ is robust over $R$ and
$R \to R'$ is a local homomorphism of complete local domains, $R'\otimes_R S$ is robust over $R'$.   But the notion
is well-behaved in many other ways. \par
Let  $R$ be a complete local domain.  We will see later that in positive characteristic both robust closure and solid closure agree with tight 
closure.  \par
If regular sequences of length 2 are $\Q$-sequences (this is an open in the non-Noetherian case for every length $> 1$),
then this is also true for latent regular sequences of length 2. Assuming this, then 
according to \cite[Theorem 12.5]{solid-closure}, in characteristic 0,  
when $\dim(R)=2$, robust closure coincides with solid closure.  We do not know whether this is true.

In dimension three robust closure can be strictly smaller than solid closure in equal characteristic 0. 
For example, in $R=\mathbb{C}[[x_1,\,x_2,\,x_3]]$, one has that $x_1^2x_2^2x_3^2\in (x_1^3,x_2^3,x_3^3)^{\bigstar}$, using the result of
\cite{Roberts-Example}. See also the second paragraph of \S\ref{intro}. However, every ideal of $R$ is its own robust closure by Theorem~\ref{regrobcl} below.

\begin{prop}\label{robcon} If  $R$ is Noetherian and $S$ is a robust $R$-algebra, then the contraction of $IS$ to $R$ is contained
in $I\ro$ for all ideals $I$ of $R$. \end{prop}

\begin{proof} Let $u \in R$ have image in  $IS$. 
Let $D$ be a complete local domain of $R$.  Then $D \otimes_RS$ is robust, and
the image of $u$ is in $I(D \otimes_R S)$, which the image of $IS$.  Hence, $u$ is in $I\ro$. \end{proof}

\begin{prop}\label{mfcon} If $R$ is Noetherian and $R\red$ contains a field, then every module-finite extension $S$ of  $R$ is robust.
Hence,  $IS \cap R \inc I\ro$ for every ideal of $R$.  Therefore, if $R$ is reduced and contains a field,  $I^+ \inc  I\ro$.  
\end{prop} 
\begin{proof} 
If $\fm$ is maximal in $R$ and $R_1$ denote the completion of  $R_{\fm}$,  then
$S_1 = R_1 \otimes S$ is a module-finite extension of $R_1$.  Let $P$ be a minimal prime of $R_1$ and let
$Q$ be a prime of $S_1$ lying over $P_1$.  Let $D = R/P_1$.  We need to show that $D \otimes S_1 = S_1/PS_1$
is robust over $D$.  But $S_1/PS_1$ maps to $S/Q$,  and $R/P = D \to S/Q$ is injective and module-finite.  
Hence, any system of parameters for $R/P$ is a system of parameters for the equicharacteristic complete local
domain $S/Q$, and the result follows from Theorem~\ref{equirob}. The final statement is immediate from the fact
that every integral extension of $R$ is a directed union of module-finite extensions. 
\end{proof}

\begin{rem} There are results for finitely generated modules corresponding to both Proposition~\ref{robcon} and 
Proposition~\ref{mfcon}.  If $N \inc M$ and we have a homomorphism $R \to S$,  we work with
the inverse image in $M$ of $\text{Im}(S \otimes_R N \to S \otimes_RM)$ under the map $M \to S  \otimes_RM$
(sending $u \mapsto 1 \otimes u$).  That is, we think of the contracted expansion of $N$ as the set of 
$u \in M$ such that $1 \otimes u$ is in the image of $S \otimes_RN$ in $S\otimes_R M$. \end{rem}
\begin{rem} Proposition~\ref{mfcon} holds for module-finite extensions of Noetherian rings that are robust
as algebras over themselves.  Thus, if Conjecture~\ref{robcj} is true, Proposition~\ref{mfcon} will hold for all Noetherian rings.
Coupled with Theorem~\ref{regrobcl} below, this would prove the direct summand conjecture. \end{rem}

\begin{prop}
\label{basic-properties}
Let $R$ be a Noetherian ring, and let $N \inc M$ and $N' \inc M'$ be Noetherian $R$-modules.  Then:
\begin{enumerate}
\item[(a)] $N\subseteq N\ro_M$.
\item[(b)] If $u \in M$,  then $u \in N\ro_M$ iff $\overline{u} \in 0\ro_{M/N}$,  where $\overline{v}$ denotes 
the image of $v \in M$ in  $M/N$.
\item[(c)] If $u \in M$, $u' \in M'$, and there is an $R$-linear map $M/N \to M/'/N'$ such that $u \mapsto u'$,
then if $u \in N\ro_M$ we have that $u' \in {N'}\ro_{M'}$.  In particular, if $N \inc N' \inc M$ then
$N\ro_M \subseteq N'\ro_{M}$, and if  $N \inc M \inc M'$ then $N\ro_M \subseteq N\ro_{M'}$.  
\item[(d)] $(N\ro_M)\ro_M=N\ro_M$.
\item[(e)] If  $N' \inc M$, then$(N\cap N')\ro_M\subseteq N\ro_M\cap N'\ro_M$. 
\item[(f)] $(N+N')\ro_M\subseteq (N\ro_M+N'\ro_M)\ro_M$.
\item[(g)] If $N_i \inc M_i$ are Noetherian, $1 \leq i \leq n$, $N = \bigoplus_i N_i$, and $M = \bigoplus_i M_i$,
then $N\ro_M$ may be identified in the obvious way with $\bigoplus_i {N_i}\ro_{M_i}$.  
\end{enumerate}
\end{prop}
\begin{proof}[Proof] (a). This is clear from the definition. \par
(b) The submodules $N'/N$ of $M/N$ are in a bijective correspondence with the submodules $N'$ of $M$ with $N \subseteq N'$,
and, by Remark~\ref{remforce} for each element $v \in M$,  a forcing algebra for $(M,N',v)$ is the same as a 
forcing algebra for $(M/N, N'/N, \overline{v})$.  Again by Remark~\ref{remforce}, 
this remains true when we make a base change to a complete local domain $R'$ of $R$,  replacing  $M$ by $R' \otimes_R M$
and $N$ by the image of $R'\otimes_RN$ in $R'\otimes_RM$. We use the notation $R$, $M$, $N$ for what we obtain
after the base change.  It follows that  ${0\rob j}_{M/N} = ({N\rob j}_M)/N$, and the result follows. \par 
(c)    It suffices to prove the result when $N=0$ and $N' =0$ (we may replace
$M, M'$ by $M/N,\, M'/N'$). The result then says that a linear map $M \to M'$ takes $0\ro_M$ into $0\ro_{M'}$.
It suffices to prove the result after base change to a complete local domain of  $R$, and so we may assume that $R$ is
a complete local domain. We then need only show that the image $0\rob j_M$  maps into $0\rob j_{M'}$ for all $j$.
We use induction on $j$.  Again, we may replace $M$ and $M'$ by quotients, $M/0{\rob j}_M$ and $M'/0{\rob j}_{M'}$, respectively.
Thus, we need only consider the case $j=1$ to do the inductive step.  But this follows from the fact that by
the third paragraph of Remark~\ref{remforce}, a forcing algebra for $(M,0,u)$ is
a forcing algebra for $(M', 0, u')$.   The corollary statements in the second sentence are
obvious. \par
(d) Let $u \in (N\ro_M)\ro_M$.  We must show that $u \in  N\ro_M$.  It suffices to consider the issue after base change to
a complete local domain of $R$.  After the base change, the image of $N\ro$ is in the robust closure of the image of $N$.
Therefore, using part (c),  it suffices to show that $(N\ro_M)\ro_M = N\ro_M$ in the case of a complete local domain $R$.
But then  $N\ro_M = {N \rob i}_M$ for $i$ sufficiently large, and so $(N\ro_M)\ro_M = {({N\rob i}_M)\rob j}_M =
{N\rob {i+j}}_M = N\ro_M$, as required.\par
(e) and (f) are immediate from the second statement in part (c) (note that $N + N' \subseteq N\ro_M + N'\ro_M$). \par 
(g) Since ${N_i}\ro_{M_i} \inc N\ro_M$ is clear from part (c),
it suffices to show that if $u  = u_1\oplus \cdots \oplus u_n\in N\ro_M$ then $u_i \in {N_i}\ro_{M_i}$ for each $i$.  Identify
$M$ with $\prod_i M_i$. The needed fact
also follows from the part (c), along with the fact that the product projection $M \to M_i$ takes $u$ to $u_i$ and $N$ onto $N_i$.  
\end{proof}

\begin{prop}\label{persist} Let $R$ be a Noetherian ring and let $N \inc M$  be finitely generated $R$-modules.
Let $S$ be a Noetherian $R$-algebra such that every maximal ideal $\fM$ of $S$ lies over a maximal ideal $\fm$
of $R$ in such a way that that the contraction of every minimal prime $\fq$ of $\widehat{S}_{\fM}$ under 
$\theta:\widehat{R}_{\fm} \to \widehat{S}_{\fM}$
is a minimal prime $\fp$ of $\widehat{R}_{\fm}$ in such a way that $\widehat{R}_{\fm}/\fp \to \widehat{S}_{\fM}/\fq$ is
an isomorphsim.  If $u \in N\ro_M$,  then $1 \otimes_u \in S \otimes M$ is in 
$\bigl(\text{Im}(S \otimes_R N \to S \otimes_R M)\bigr)\ro_{S \otimes_R M}$.\par
In particular, the conclusion holds when $S$ is the localization of $R$
at one or at finitely many maximal ideals, the completion of $R_{\fm}$ for a maximal ideal $\fm$ of $R$, 
or when $S$ is the quotient of $R$ by a minimal prime.
\end{prop}
\begin{proof}  The hypothesis at once implies that every complete local domain of $S$ is a complete local domain of $R$, which 
can be checked instead for the map $R_{\fm} \to S_{\fM}$ or for the induced map of completions.  This immediately yields the
first conclusion.\par
We next consider the statement of the second paragraph.  In the case of localization at a finite set of maximal ideals, each
localization $S_{\fM}$ is isomorphic to the corresponding $R_{\fm}$, from which the result is obvious.  If we consider  $R \to \widehat{R}_{\fm}$,
the induced map of completions is the identity on $\widehat{R}_{\fm}$.  Therefore we need only consider the case where $S = R/P$
for a minimal prime $P$ of $R$.  Then $\fM = \fm/P$ for a maximal ideal of $\fm$ of $R$, and we can reduce to studying the local
case, i.e., we may assume that $(R,\,\fm)$ is local and $S = R/P$ for $P$ minimal.  Then $\widehat{S} \cong \widehat{R}/P\widehat{R}$,
and it suffices to observe that every minimal prime $\fq$ of $P\widehat{R}$ is a minimal prime of $\widehat{R}$, by the discussion
in the paragraph following Propopsition~\ref{two-generic-forcing-algebras}.   \end{proof}

\begin{prop} Let $R$ be a Noetherian ring, let $\fN$ denote the radical of the ideal $(0)$ of $R$,  i.e., the nilradical of $R$, let
$R\red = R/\fN$, and let $N \inc M$ be finitely generated $R$-modules.
\begin{enumerate}
\item[(a)] The robust closure of the ideal $(0)$ is the nilradical $\sqrt{(0)}$ of $R$. The robust closure 
$0\ro_M$ of $0$ in $M$ is contained $\sqrt{(0)}M$.
\item[(b)] Let $R\red=R/\sqrt{(0)}$. Then $N\ro_M$ is the inverse image of the robust closure of 
$\im(R\red\otimes_RN\to R\red\otimes_RM)$ in $R\red\otimes_RM=M/\sqrt{(0)}M$.
\item[(c)] An element $u\in M$ is in $N\ro_M$ if and only if for every minimal prime $\fp$ of $R$, 
the image of $u$ in $M/\fp M$ is in the robust closure of $\im(N/\fp N\to M/\fp M)$ in $M/\fp M$, computed over $R/\fp $.
\end{enumerate}
\end{prop} 
\begin{proof} (a) First note that $\fN$ maps to 0 in every complete local domain of $R$, which shows that $\fN \inc (0)\ro$. 
If $r \in R$ is not nilpotent choose a minimal prime $\fp$ of $R$ that does not contain $r$ and a maximal ideal $\fm$ of $R$ 
that contains $\fp$.  Then $r \notin \fp R_{\fm}$ and so $r$ is not nilpotent in $R_{\fm}$.  It follows that the image of $r$ is not 
nilpotent in the completion of $R_{\fm}$,  and so it is not in a minimal prime of the completion.  It therefore suffices to show 
that $(0)$ is solidly closed in a complete local domain $R$.  
This is clear because every nonzero element $x$ in a complete local
domain is a parameter or a unit:  the generic forcing algebra for $(R, (0), x)$ is a polynomial ring over $R/xR$,  and $R/xR$ is obviously
not a robust $R$-algebra). Parts (b) and (c) are clear from Proposition~\ref{persist} and the fact that every map 
$R \to D$, where $D$ is a complete local domain of $R$, factors $R\red \to D$ (respectively,  $R/\fp \to D$) for some
minimal prime $\fp$ of $R$ in such a way that $D$ is a complete local domain of $R\red$ (respectively, of $R/\fp$).  
\end{proof}

\begin{prop} For all ideals $I$ of a Noetherian ring $R$,  $I\ro\subseteq I^{\bigstar}\subseteq \overline{I}$.
\end{prop}
\begin{proof} The fact that $I^{\bigstar}\subseteq \overline{I}$ is  \cite[Theorem 5.10]{solid-closure}. Thus, we only need to show
that if $u \in I\ro$ then $u \in I^{\bigstar}$, and it suffices to show this after base change to a complete local domain
of $R$, where the issue becomes showing that $I\rob j \inc I^{\bigstar}$.  By induction on $j$ this reduces to the
case where $j =1$,  where it follows from the fact that a robust algebra over a complete local domain is solid:  see the
comment following Definition~\ref{defn-robust-algebra}.
\end{proof}

\begin{prop} Let $R$ be a Noetherian ring.
If $I$ is principal and every module-finite extension of a complete local domain $D$ of $R$ is a robust $D$-algebra,
then $I\ro=\overline{I}$.  In particular, if  $R\red$ contains a field or $R$ has dimension at most 2 and $I$ is principal,
then $I\ro = \overline{I}$. \end{prop}

\begin{thm}\label{regrobcl} Let $R$ be a regular Noetherian ring.  Then for any two finitely generated $R$-modules $N \inc M$,
$N\ro_M = N$.  \end{thm}
\begin{proof} Suppose $u \in M$ is such that $u \in N\ro_M \backslash N$.  Localize at a maximal ideal $\fm$ of $R$ such that
$u \notin N_{\fm}$ and complete. Thus, we may assume without loss of generality that $R$ is a complete regular local ring.

Choose $N' \supseteq N$ maximal in $M$ with respect to 
not containing $u$.  Then by Proposition~\ref{basic-properties}(d) we still have $u \in {N'}\ro_M$.  But then  
the image of $u$ is in every submodule of $M/N'$.
We may replace $M,\,N$, and $u$ by $M/N',\, 0$ and the image of $u$ in $M/N'$.  Hence, we may assume that
$M$ has finite length,  $N = 0$, and that $u$ generates the socle.  Let $\vect x d$ be a regular system of parameters for $R$.
Then the injective hull of $R/\fm$ over $R$ is the directed union of modules of the form $R/I_t$ where $I = (x_1^t,\, \ldots,\, x_d^t)R$, 
and $M$ is contained in $R/I_t$ for $t \gg 0$.  Moreover, up to a unit factor, $u$ must be the socle generator 
$(x_1 \cdots x_d)^{t-1}$ for $R/I_t$.   It will therefore suffice to show that $0 = 0\ro_{R/I_t}$, and for this it will suffice
to show that $0\roi = 0$.  This will follow if we can show that there is no forcing algebra that is robust for
the triple $(R/I_t, 0, v)$ where $v$ is a nonzero element of $R/I_t$.  Note that any such forcing algebra must also
be a forcing algebra for $(R/I_t, 0, u)$,  since $u$ is a multiple of $v$ in $R/I_t$.   It therefore suffices to show that the
generic forcing algebra $S  =R[\vect Z d]/(x_1^{t-1}\cdots x_d^{t-1} - \sum_{i=1}^d Z_ix_i^t)$ is not robust over $R$.  But
this is clear, $S/(x_1^t, \cdots, \, x_d^t)$ will have a filtration using ideals generated by monomials in the $x_i$ of
length $t^d-1$, and so $\vect x d$ is not a Q-sequence in $S$.   \end{proof}

\begin{prop}
Let $R\to S$ be a homomorphism of Noetherian. Suppose that 
\begin{enumerate}
\item[($\dagger$)] $\MaxSpec(S) \to \MaxSpec(R)$
is surjective and that for every minimal prime $\fp$ of $R$ and maximal ideal
$\fm$ of $R$ with $\fp \inc \fm$,  ht$(\fm S/\fp S) \geq $ht$(\fm/\fp)$. 
\end{enumerate}  
Let $N \inc M$ be finitely generated $R$-modules.  If $u \in M$ is such that $1 \otimes u$ is in the robust closure
of the image of $S\otimes_RN$ in $S\otimes_R N$,  then $u \in N\ro_M$.  In particular,  for an ideal $I \inc R$,
the contraction to $R$ of $(IS)\ro$, calculated over $S$, is contained in $I\ro$.\par
The hypothesis $(\dagger)$ holds if $S$ is faithfully flat over $R$, or if $R$ is universally catenary and $S$ is an integral extension of $R$,  
or if $R \to S$ is a local homomorphism,  $R$ is equidimensional,
and the image of one system of parameters of $R$ is a system of parameters in $S$.  
\end{prop}
\begin{proof} It will suffice to show that for every complete local domain $D$ of $R$ there is a complete local domain $D'$ of $S$
such that $R \to S \to D'$ factors  $R \to D \to D'$ and every system of parameters for $D$ is part of a system of parameters for $D'$.
It follows that every robust $D'$-algebra is a robust $D$-algebra and we may apply Lemm
Assume that $I=(g_1,\dots,g_n)$. Since $\varphi(u)\in (IS)\ro$, each system of parameters of $S$ is a $\Q$-sequence in $\frac{S[Z_1,\dots,Z_n]}{(\varphi(u)-\sum^n_{i=1}
\varphi(g_i)Z_i)}$. It is clear that $\varphi:R\to S$ induces a ring homomorphism 
\[\varphi':\frac{R[Z_1,\dots,Z_n]}{(u-\sum^n_{i=1}
g_iZ_i)}\to \frac{S[Z_1,\dots,Z_n]}{(\varphi(u)-\sum^n_{i=1}
\varphi(g_i)Z_i)}.\]
Since the image of each system of parameters of $R$ becomes a (partial) system of parameters of $S$, we know that each system of parameters is also a $\Q$-sequence in $\frac{R[Z_1,\dots,Z_n]}{(u-\sum^n_{i=1}g_iZ_i)}$. This proves that $u\in I\ro$.
\end{proof}

\section{Comparison with tight closure in positive characteristic}\label{charp} 
In this section, we compare robust closure with tight closure in positive characteristic. In particular, we prove the following theorem.
\begin{thm}
\label{solid-implies-content-1}
Let $(R,\fm)$ be a $d$-dimensional complete local domain of characteristic $p>0$. Let $I$ be an ideal of $R$, $u$ an element of $R$, and $S$ a generic forcing algebra for $(R,I,u)$. Then $H^d_{\fm}(S)\neq 0$ if and only if the image of each system of parameters for $R$ in $S$ forms a $\Q$-sequence in $S$.
\end{thm}

\begin{proof}[Proof]
The sufficiency is clear.\par
Assume that $H^d_{\fm}(S)\neq 0$. From the local cohomology criterion for solid closure (\cite{solid-closure}), we have that $u\in I^{\bigstar}$. And by \cite[Theorem 8.6(b)]{solid-closure} $I^{\bigstar}=I^*$ under our hypotheses, hence $u$ is an element of $I^*$, i.e., for $q\gg 0$, we have $cu^q\in I^{[q]}=(g^q_1,\dots,g^q_n)$ where $\{g_1,\dots,g_n\}$ is a set of generators of $I$. We need to prove that the image of a system of parameters $x_1,\dots,x_d$ in $S=\frac{R[Z_1,\dots,Z_n]}{u-\sum^n_{l=1}Z_lg_l}$ forms a $\Q$-sequence in $S$. Assume otherwise, then there would exist an integer $q_0=p^{e_0}$ and a short filtration of $\frac{S}{(x^{q_0}_1,\dots,x^{q_0}_d)}$ with $h<q^d_0$ factors. This short filtration would be given by $\tilde{s}_1,\dots,\tilde{s}_h\in S$ ($\tilde{s}_h=1$) satisfying
\begin{align}
(x_1,\dots,x_d) \tilde{s}_1 &\subseteq (x^{q_0}_1,\dots,x^{q_0}_d)\notag\\
(x_1,\dots,x_d) \tilde{s}_{i+1}&\subseteq (\tilde{s}_1,\dots,\tilde{s}_i)+(x^{q_0}_1,\dots,x^{q_0}_d), 1\leq i\leq h-1\notag
\end{align} 
Pick and fix $s_i$ a lifting of $\tilde{s}_i$ in $R[Z_1,\dots,Z_n]$. Then there would be the following equations
\begin{equation}
\label{filtration-equations}
x_js_i=\sum^{i-1}_{k=1}\alpha_{ijk}s_k+\sum^d_{k=1}\beta_{ijk}x^{q_0}_k+\gamma_{ij}(u-\sum^n_{l=1}Z_lg_l),\ \alpha_{ijk},\beta_{ijk},\gamma_{ij}\in R[Z_1,\dots,Z_n] 
\end{equation}
where $\alpha_{1jk}=0$. Since $cu^q\in I^{[q]}=(g^q_1,\dots,g^q_n)$, we have $c^{\frac{1}{q}}u=\sum^n_{l=1}r^{\frac{1}{q}}_lg_l$ in $R^{\frac{1}{q}}$ for some $r_i\in R$. After the substitution $Z_l=c^{-\frac{1}{q}}r^{\frac{1}{q}}_l$ in each equation appearing in (\ref{filtration-equations}), we set $N$ to be the least integer such that $c^{\frac{N}{q}}$ can clear all denominators appearing in all equations in (\ref{filtration-equations}). Next we define $\varphi:R[Z_1,\dots,Z_n]\to R^{\frac{1}{q}}$ by $Z_l\mapsto c^{-\frac{1}{q}}r^{\frac{1}{q}}_l$ and $r\mapsto r$ for all $r\in R$. Set $\sigma_i=\varphi(s_i)c^{\frac{iN}{q}}\in R^{\frac{1}{q}}$ for $1\leq i\leq h$. Then we have the following equations 
\begin{equation}
\label{sigma-equation}
x_j\sigma_i=c^{\frac{(i-1)N}{q}}\varphi(\alpha_{ij1})\sigma_1+\cdots+c^{\frac{N}{q}}\varphi(\alpha_{ij(i-1)})\sigma_{i-1}+c^{\frac{iN}{q}}\sum^d_{k=1}\varphi(\beta_{ijk})x^{q_0}_k
\end{equation}
Set $$\tau_i=\begin{cases} 
c^{\frac{hN}{q}}\sigma_{\frac{i+1}{2}},\ {\rm when}\ i\ {\rm odd\ and}\ i<2h+1\\
\sigma_{\frac{i}{2}},\ {\rm when}\ i\ {\rm even}\\
1,\ {\rm when}\ i=2h+1 
\end{cases}$$
and $M_j=R^{\frac{1}{q}}\tau_1+\cdots+R^{\frac{1}{q}}\tau_j+(x^{q_0}_1,\dots,x^{q_0}_d)R^{\frac{1}{q}}$ for $j=1,\dots,2h+1$, then by equations in (\ref{sigma-equation}) we have a filtration of $\frac{R^{\frac{1}{q}}}{(x^{q_0}_1,\dots,x^{q_0}_d)R^{\frac{1}{q}}}$ as
\[0=M_0\subseteq \frac{M_1}{(x^{q_0}_1,\dots,x^{q_0}_d)R^{\frac{1}{q}}}\subseteq \cdots\subseteq \frac{M_{2h+1}}{(x^{q_0}_1,\dots,x^{q_0}_d)R^{\frac{1}{q}}}=\frac{R^{\frac{1}{q}}}{(x^{q_0}_1,\dots,x^{q_0}_d)R^{\frac{1}{q}}} \]
in which $\frac{M_{j+1}}{M_j}$ is killed by $(c^{\frac{hN}{q}},x^{q_0}_1,\dots,x^{q_0}_d)$ when $j$ is even and it is killed by $(x_1,\dots,x_d)$ when $j$ is odd. Consequently, we have
\[\lambda(\frac{R^{\frac{1}{q}}}{(x^{q_0}_1,\dots,x^{q_0}_d)R^{\frac{1}{q}}})\leq h\lambda(\frac{R^{\frac{1}{q}}}{(x_1,\dots,x_d)R^{\frac{1}{a}}})+\lambda(\frac{R^{\frac{1}{q}}}{(c^{\frac{hn}{q}}x^{q_0}_1,\dots,x^{q_0}_d)R^{\frac{1}{q}}})\] 
which is equivalent to 
\[\lambda(\frac{R}{(x^{qq_0}_1,\dots,x^{qq_0}_d)})\leq h\lambda(\frac{R}{(x^{q}_1,\dots,x^q_d)})+(h+1)\lambda(\frac{\frac{R}{(c^{\frac{hN}{q}})}}{(x^{qq_0}_1,\dots,x^{qq_0}_d)\frac{R}{(c^{\frac{hN}{q}})}})\]
and hence
\begin{equation}
\label{length-inequality}
\frac{\lambda(\frac{R}{(x^{qq_0}_1,\dots,x^{qq_0}_d)})}{(qq_0)^d}\leq h\frac{\lambda(\frac{R}{(x^{q}_1,\dots,x^q_d)})}{(qq_0)^d}+(h+1)\frac{\lambda(\frac{\frac{R}{(c^{\frac{hN}{q}})}}{(x^{qq_0}_1,\dots,x^{qq_0}_d)\frac{R}{(c^{\frac{hN}{q}})}})}{(qq_0)^d}
\end{equation}
Letting $q\to \infty$, we have
\[e_{\underline{x}}(R)\leq \frac{h}{q^d_0}e_{\underline{x}}(R)+0\]
a contradiction since $h<q^d_0$, where $e_{\underline{x}}(R)$ denotes the multiplicity of $R$ with respect to the ideal $(x_1,\dots,x_d)$ and the limit of the second summand in the right hand side in (\ref{length-inequality}) is 0 because $\dim(\frac{R}{(c^{\frac{hN}{q}})})<\dim(R)$. This completes the proof.
\end{proof}

The following corollary follows immediately from Theorem \ref{solid-implies-content-1}
\begin{cor}
\label{cor: tight and robust closures coincide in char p}
Let $R$ be a Noetherian ring of characteristic $p>0$ and assume that $R$ has a completely stable test element $c$, then for all ideals $I$ of $R$ one has $I_j=I^*$ for all $j\geq 1$, where $I_j$ is as in Definition \ref{defn-robust-closure}. In particular, \[I\ro=I^*\]
for all ideals $I$.\par
Let $u$ be an lement of $R$, if one system of parameters $x_1,\dots,x_d$ for $R$ forms a $\Q$-sequence in the generic forcing algebra $S=\frac{R[Z_1,\dots,Z_n]}{(u-\sum^n_{l=1}Z_lg_l)}$ where $\{g_1,\dots,g_n\}$ is a set of generators of $I$, so does every system of parameters for $R$.
\end{cor}

\section{Q-sequences, superheight and Lyubeznik's question}\label{Qsup}  

In this section, we investigate the connections between $\Q$-sequences and superheight. First we recall the definition of superheight.

\begin{defn}[superheight]
Let $R$ be a Noetherian commutative ring and $I$ an ideal of $R$. The {\it superheight} of $I$, denote by $\supht(I)$, is defined as
\[\sup\{\hgt_S(IS)|{\rm for\ all\ Noetherian\ }R{\rm -algebras\ }S\}.\]
\end{defn}

\begin{rem}
\label{superheight-implies-Q-sequence}
It follows immediately from \cite[Theorem 4.7]{quasi-length} that, if $R$ contains a field and $\supht((x_1,\dots,x_n))=n$, then $x_1,\dots,x_n$ form a $\Q$-sequence.
\end{rem}

Given Remark \ref{superheight-implies-Q-sequence}, it is natural to ask:
\begin{question}
\label{Q-sequence-implies-superheight}
Assume that $R$ contains a field and $x_1,\dots,x_n$ is a $\Q$-sequence. Is it true that $\supht((x_1,\dots,x_n))=n$?
\end{question}

As we will see from the examples in this section, the answer to Question \ref{Q-sequence-implies-superheight} is false. We will give two examples, one in positive characteristic and the other in characteristic 0.

\begin{ex}[Brenner-Monsky Example]
Let $A=\frac{\bar{\mathbb{F}}_2(t)[x,y,z]}{(z^4+xyz^2+x^3z+y^3z+tx^2y^2)}$. As shown in \cite[Remark 4.1]{Brenner-Monsky}, the element $x^3y^3$ is contained in the tight closure of $(x^4,y^4,z^4)$, but not in the plus closure of $(x^4,y^4,z^4)$. Since the sequence $x,y$ is a system of parameters in $A$, it follows immediately from Theorem \ref{solid-implies-content-1} that the sequence $x,y$ forms a $\Q$-sequence in the forcing algebra \[R=\frac{\bar{\mathbb{F}}_2(t)[x,y,z, u,v,w]}{(z^4+xyz^2+x^3z+y^3z+tx^2y^2, x^3y^3-ux^4-vy^4-wz^4)}.\]
However, since $x^3y^3$ is not in the plus closure of $(x^4,y^4,z^4)$, we claim that the superheight of $(x,y)$ as an ideal in $R$ is 1 and we reason as follows. Assume that $\supht((x,y))=2$. According to a theorem by Koh (\cite[Theorem 1]{KohSuperheight}), since $R$ is a finite algebra over a field, the finite superheight of $(x,y)$ is also 2, {\it i.e.}, there exists a module-finite $R$-algebra $R'$ such that $\hgt((x,y)R')=2$. After localizing $R'$ at a height-2 minimal prime over $(x,y)R'$, we may assume that the dimension of $R'$ is 2 and $x,y$ is a system of parameters in $R'$. Since $x,y$ is also a system of parameters in $A$, we see that $R'$ is a module-finite $A$-algebra. Since $x^3y^3\in (x^4,y^4,z^4)$ in $R$ and $R'$ is an $R$-algebra, we have $x^3y^3\in (x^4,y^4,z^4)$ in $R'$. This implies that $x^3y^3$ is in the plus closure of $(x^4,y^4,z^4)$ in $A$, a contradiction.
\end{ex}

\begin{prop}
\label{forcing-algebra-cubic}
Let $R=\mathbb{C}[[x,y,z,u,v]]/(x^3+y^3+z^3,z^2-ux-vy)$. Then the sequence $x,y$ is a Q-sequence, but $\supht((x,y))=1$.
\end{prop}
\begin{proof}
It is clear that $R$ is the forcing algebra of the ring $A=\frac{\mathbb{C}[[x,y,z]]}{(x^3+y^3+z^3)}$, the ideal $I=(x,y)$ in $A$ and the element $z^2$ in $A$. Now $x,y$ is a Q-sequence in $R$ follows from the fact that in $A$ the element $z^2$ belongs to the big equational tight closure of $I$ ({\it cf.} \cite[Example 12.7]{solid-closure}). 

Assume that $\supht((x,y))$ is 2 and we seek for a contradiction. Let $S$ be a Noetherian $R$-algebra in which the height of $(x,y)S$ is 2. After localizing and completing $S$ at a height-2 prime ideal containing $(x,y)S$, we may assume that $S$ is a 2-dimensional complete local ring. Let $T=\mathbb{C}[[x,y]]$ be the subring of $S$ generated by $x,y$ over $\mathbb{C}$. Since $x,y$ is a system of parameter of $S$, we know that $T$ is a regular local ring. Still denote the image of $z$ in $S$ by $z$. Let $T'=T[z]\subseteq S$. Since $z$ satisfies $x^3+y^3+z^3=0$ in $R$, it still satisfies the equation in $S$. Hence there is a surjection $A\twoheadrightarrow T'$. Since $A$ is a 2-dimensional domain and the dimension of $T'$ is also 2, we must have $A\cong T'$. Since $A$ is normal, so is $T'$. Since in $R$ we have $z^2=ux+vy$, the same equation holds in $S$. This implies that the element $z^2$ in $T'$ is in $(x,y)S\cap T'$. Since $T'$ is normal, $z^2$ must be contained in $(x,y)$ in $T'\cong A$. However, this is impossible since $z^2\notin (x,y)$ in $A$, a contradiction.
\end{proof}

It turns out that Proposition \ref{forcing-algebra-cubic} can be used to give a negative answer to part (2) of the following question raised by Gennady Lyubeznik in \cite{Lyubeznik-survey}.
\begin{question}[page 144 in \cite{Lyubeznik-survey}]
\label{lyu-question}
Let $(R,\fm)$ be a complete local domain with a separably closed residue field.
\begin{enumerate}
\item Find necessary and sufficient condition on $I$ such that $H^j_I(M)=0$ for all integers $j\geq \dim(R)-1$ and all $R$-modules $M$.
\item Let $I$ be a prime ideal. Is it true that $H^j_I(M)=0$ for all integers $j\geq \dim(R)-1$ and all $R$-modules $M$ if and only if $(P+I)$ is not primary to the maximal ideal for any prime ideal $P$ of height 1?
\end{enumerate}
\end{question}

\begin{prop}
\label{prop: forcing-algebra-cubic}
Let $R=\frac{\mathbb{C}[[x,y,z,u,v]]}{(x^3+y^3+z^3,z^2-ux-vy)}$ and $I=(x,y,z)$. Then
\begin{enumerate}
\item $(P+I)$ is not primary to the maximal ideal of $R$ for any prime ideal $P$ of height 1; and
\item $H^2_I(R)\neq 0$.
\end{enumerate}
In particular, the answer to Question \ref{lyu-question}(2) is negative.
\end{prop}
\begin{proof}Since we have proved that the superheight of $(x,y)$ in the ring $$R=\frac{\mathbb{C}[[x,y,z,u,v]]}{(x^3+y^3+z^3,z^2-ux-vy)}$$ is 1, it is clear that $(x,y,z)+P$ is not $\fm$-primary for any height-1 prime ideal $P$ of $R$ (otherwise, in the ring $R/P$, the height of $(x,y)$ would be 2). However, we claim that $H^2_{(x,y,z)}(R)\neq 0$. One can either see this from the fact that in $A$ the element $z^2$ is contained in the solid closure of $(x,y)$, or verify directly that the class $[\frac{1}{xy}]$ is not 0 in $H^2_{(x,y,z)}(R)$.
\end{proof}

\begin{rem}
Note that, in the ring $R_I$, the ideal $IR_I$ is 2-generated; hence $R$ is not normal by Serre's criterion ($I$ is a height-1 prime). Or, by the Jacobian criterion, one can see that the singular locus is actually defined by $I$ and hence $R$ is not normal.  To obtain a counterexample to Question \ref{lyu-question}(2) that is a normal domain, we consider the normalization $R'$ of $R$. 

It is straightforward to check that the fraction $w=\frac{y^2+zv}{x}=\frac{-x^2-zu}{y}$ satisfies the equation
\[w^3-xzu+2yzv+y^3-u^3+v^3=0\]
and hence $w$ is integral over $R$. Also one can check that $R'=R[w]$, i.e.,
$R'=\mathbb{C}[[x,y,z,u,v,w]]/\fP,$ where $\fP$ is the ideal of $\mathbb{C}[[x,y,z,u,v,w]]$ generated by the elements
$z^2-xu-yv$, $yw+x^2+zu$, $xw-y^2-zv,zw^2+xyz+yu^2+xv^2$, and $w^3-xzu+2yzv+y^3-u^3+v^3$.
By an argument similar to the one above, we see that, in $R'$, we have that $I+Q$ is not primary to the maximal ideal for each height one 
prime $Q$. We claim that $H^2_I(R')\neq 0$ and we reason as follows. Since $z$ is contained in the radical of $(x,y)R'$, we know that $H^t_I(M)=H^t_{(x,y)}(M)$ for all integers $t$ and all $R'$-modules $M$. It is clear that $xw$ and $yw$ are in $R$. 
Thus, $(x,y)\frac{R'}{R}=0$ since $R'=R[w]$. 
This implies that $R'/R$ is $(x,y)$-torsion, therefore $H^t_{(x,y)}(R'/R)=0$ for all $t\geq 1$. 
In particular, $H^1_{(x,y)}(R'/R)=H^2_{(x,y)}(R'/R)=0$. Now the short exact sequence $0\to R\to R'\to R'/R\to 0$ 
induces an exact sequence of local cohomology modules 
\[\cdots \to H^1_{(x,y)}(R'/R) \to H^2_{(x,y)}(R) \to H^2_{(x,y)}(R') \to H^2_{(x,y)}(R'/R) \to \cdots\]
Therefore, $H^2_{I}(R')=H^2_{(x,y)}(R')\cong H^2_{(x,y)}(R)=H^2_{I}(R)\neq 0$. 
\end{rem}

\begin{rem}
An immediate consequence of \cite[Theorem 2.5]{Huneke-Lyubeznik-vanishing-LC} is the following

{\it Let $(R,\fm)$ be a $d$-dimensional equal-characteristic complete local domain with a separably closed residue field. Let $I$ be a prime ideal of $R$. Assume that $\embdim(R)-\mdim(R/I)<d-1$. Then 
\[H^{d-1}_I(M)=H^d_I(M)=0\]
for all $R$-modules $M$.}

In Proposition \ref{prop: forcing-algebra-cubic}, $R$ is a 3-dimensional complete intersection where $\embdim(R)-\mdim(R/I)=2=d-1$ and $H^2_I(R)\neq 0$. This indicates that, without further assumptions on $R$ and $I$, the bound $\embdim(R)-\mdim(R/I)<d-1$ is sharp. 
\end{rem}


\section{Questions, conjectures, and related problems}\label{qu}  

In this section we have collected some questions and conjectures related to the
behavior of quasilength and content.  

\begin{conj} In a complete local domain $R$ of mixed characteristic, every
system of parameters is a Q-sequence (equivalently, $R$ is robust as an $R$-algebra). \end{conj}

The Conjecture above implies the direct summand conjecture at once ({\it cf.} \cite[Remark 4.8]{quasi-length}).


\begin{question} Let $R$ be a complete local domain of mixed characteristic, and
$z_1, \, \ldots, \, x_d$ a system of parameters.  Is it true that $(x_1, \, \ldots, \, x_k)R:_R x_{k+1}
\subseteq \bigl((x_1, \, \ldots, \, x_k)R\bigr)\ro$.  That is, does robust closure have the
colon-capturing property?  \end{question}

\begin{question} If $R \to S$ is a local homomorphism of complete local domains
and $r \in I\ro$,  where $I \subseteq R$,  is the image of $r$ in $S$ in $(IS)\ro$?
An affirmative answer gives a form of persistence for robust closure.  \end{question}

Affirmative answers to these questions would imply the following Conjecture, whose
statement does not refer to robust closure.

\begin{conj}\label{capture}
Let $(R,\fm)$ be  $d$-dimensional complete local domain of mixed characteristic $p$. Assume that $p,x_2,\dots,x_d$ is a system of parameters. Let $\overline{R} := R/pR$.   
Let $J = (x_2,\,\ldots,\,x_d)R$, and let $J_n = (x_2,\dots,x_d):_R p^n \subseteq R$.  Then for all $n > 0$, 
$J_n\overline R \subseteq (J\overline{R})^*$.
\end{conj}

The point is that $J_n$ would be in $J\ro$,  and this would persist when we map to
$R/pR$.  But in $R/pR$,  robust closure coincides with tight closure.  We prove below
that Conjecture~\ref{capture} suffices to prove the direct summand conjecture.  We also
prove that it holds in dimension at most 3.  This argument in dimension 3 needs very difficult results of Heitmann (\cite[Theorem 0.2]{HeitmannExtendedPlusClosure}).
However, which can be used to prove the direct summand conjecture in several ways.   

\begin{thm}
Conjecture \ref{capture} implies the Monomial Conjecture (or equivalently the Direct Summand Conjecture).
\end{thm}
\begin{proof}
Without loss of generality, we may assume that $R$ is a complete local domain of mixed characteristic. Assume the Monomial Conjecture does not hold in $R$, {\it i.e.} there is a system of parameters $p=x_1,\dots,x_d$ of $R$ such that $(x_1,\cdots x_d)^{t-1}\in (x^t_1,\dots,x^t_d)$ for some integer $t\geq 2$ (we may choose $x_1=p$ by \cite[page 539]{HochsterCanonicalElement}). This will imply that there are elements $r_1,\dots,r_d$ such that \[x^{t-1}_1((x_2\cdots x_d)^{t-1}-r_1x_1)=\sum^d_{i=2}r_ix^t_i.\] 
Hence $(\bar{x}_2\cdots \bar{x}_d)^{t-1}\in (\bar{x}_2,\dots,\bar{x}_d)^*$ in $\bar{R}:=R/x_1R=R/pR$, where $\bar{x}_i$ denotes the image of $x_i$ in $\bar{R}$. Then by \cite[Theorem 8.6(b)]{solid-closure}, $H^{d-1}_{(x_2,\dots,x_d)}(S)\neq 0$ where $S=\frac{\bar{R}[Z_2,\dots,Z_d]}{((\bar{x}_2\cdots \bar{x}_d)^{t-1}-\sum^d_{i=2}Z_i\bar{x}^t_i)}$. However, since $\bar{R}$ and hence $S$ contain a field of characteristic $p$, we have $[\frac{1}{(x_2\cdots x_d)^{p^e}}]=[\frac{1}{(x_2\cdots x_d)}]^{p^e}=0\in H^{d-1}_{(x_2,\dots,x_d)}(S)$; consequently $H^{d-1}_{(x_2,\dots,x_d)}(S)=0$, a contradiction.
\end{proof}

\begin{thm}
Conjecture \ref{capture} holds in dimension 3.
\end{thm}
 \begin{proof}
 Let $(R,\fm)$ be a 3-dimensional complete local domain of mixed characteristic. Let $p,x,y$ be a system of parameters of $R$ and $J_n=(x,y):_Rp^n$. it suffices to show that $J_n\overline{R}\subseteq (J\overline{R})^*$, where $\overline{R}=R/\fp$, for each minimal prime ideal $\fp$ of $p$. Assume that $u\in J_n$ for some $n$, {\it i.e.} $up^n\in (x,y)$ for some $n$, then by \cite[Theorem 0.2]{HeitmannExtendedPlusClosure}, for any nonzero element $c\in \fm$ and every $q=p^e$ one has $c^{1/q}u\in (x,y)R^+$. This holds true after modding out $\fp$, and therefore we have in $\overline{R}:=R/\fp$ that $\bar{c}^{1/q}\bar{u}\in (\bar{x},\bar{y})\overline{R}^+$. Then it follows from \cite[Theorem 3.1]{HochsterHunekeElementsSmallOrders} that $\bar{u}\in (\bar{x},\bar{y})^*=(J\overline{R})^*$. The finishes the proof.
 \end{proof}
 
\begin{question} Is it always true that 
$\underline{\fh}^d_{\underline{x}}(R)=\fh^d_{\underline{x}}(R)$? 
This is true in positive characteristic as shown in \cite[Theorem 3.9]{quasi-length}. 
We do not even know whether the weaker statement that if 
$\underline{\fh}^d_{\underline{x}}(R)=0$ then $\fh^d_{\underline{x}}(R)=0$ is true.\end{question}

\begin{question} Let $R$ be regular  of characteristic $p>0$, let $I$ be a prime ideal with 
$x_1,\,\ldots,\, x_n$ a minimal set of generators. If $h<n$ , is it true that $\fh^n_{\underline{x}}(R)=0$?
\end{question}

\begin{question}
Assume that 
\[S_1=R[Y_1,\dots,Y_n]/(u-\sum^n_{l=1}Y_lg_l)\ {\rm and}\ S_2=R[Z_1,\dots,Z_n]/(v-\sum^n_{l=1}Z_lg_l)\]
satisfy $\fh^d_{\underline{x}}(S_1)=\fh^d_{\underline{x}}(S_2)=1$, then do we have
\[\fh^d_{\underline{x}}(S_1\otimes_RS_2)=1?\]
\end{question}


\bibliographystyle{skalpha}
\bibliography{CommonBib}

\def\cprime{$'$} \def\cprime{$'$}
  \def\cfudot#1{\ifmmode\setbox7\hbox{$\accent"5E#1$}\else
  \setbox7\hbox{\accent"5E#1}\penalty 10000\relax\fi\raise 1\ht7
  \hbox{\raise.1ex\hbox to 1\wd7{\hss.\hss}}\penalty 10000 \hskip-1\wd7\penalty
  10000\box7}
\providecommand{\bysame}{\leavevmode\hbox to3em{\hrulefill}\thinspace}
\providecommand{\MR}{\relax\ifhmode\unskip\space\fi MR}
\providecommand{\MRhref}[2]{%
  \href{http://www.ams.org/mathscinet-getitem?mr=#1}{#2}
}
\providecommand{\href}[2]{#2}
\begin{thebibliography}{Koh88}

\bibitem[Bre03]{BrennerParasolidClosure}
{\sc H.~Brenner}: \emph{How to rescue solid closure}, J. Algebra \textbf{265}
  (2003), no.~2, 579--605. {\sf\scriptsize 1987018 (2004e:13007)}

\bibitem[BM10]{Brenner-Monsky}
{\sc H.~Brenner and P.~Monsky}: \emph{Tight closure does not commute with
  localization}, Ann. of Math. (2) \textbf{171} (2010), no.~1, 571--588.
  {\sf\scriptsize 2630050 (2011d:13005)}

\bibitem[Hei05]{HeitmannExtendedPlusClosure}
{\sc R.~C. Heitmann}: \emph{Extended plus closure and colon-capturing}, J.
  Algebra \textbf{293} (2005), no.~2, 407--426. {\sf\scriptsize 2172347
  (2007d:13005)}

\bibitem[Hoc83]{HochsterCanonicalElement}
{\sc M.~Hochster}: \emph{Canonical elements in local cohomology modules and the
  direct summand conjecture}, J. Algebra \textbf{84} (1983), no.~2, 503--553.
  {\sf\scriptsize 723406}

\bibitem[Hoc94]{solid-closure}
{\sc M.~Hochster}: \emph{Solid closure}, Commutative algebra: syzygies,
  multiplicities, and birational algebra ({S}outh {H}adley, {MA}, 1992),
  Contemp. Math., vol. 159, Amer. Math. Soc., Providence, RI, 1994,
  pp.~103--172. {\sf\scriptsize 1266182 (95a:13011)}

\bibitem[HH91]{HochsterHunekeElementsSmallOrders}
{\sc M.~Hochster and C.~Huneke}: \emph{Tight closure and elements of small
  order in integral extensions}, J. Pure Appl. Algebra \textbf{71} (1991),
  no.~2-3, 233--247. {\sf\scriptsize 1117636 (92i:13002)}

\bibitem[HH09]{quasi-length}
{\sc M.~Hochster and C.~Huneke}: \emph{Quasilength, latent regular sequences,
  and content of local cohomology}, J. Algebra \textbf{322} (2009), no.~9,
  3170--3193. {\sf\scriptsize 2567415 (2011a:13036)}

\bibitem[HL90]{Huneke-Lyubeznik-vanishing-LC}
{\sc C.~Huneke and G.~Lyubeznik}: \emph{On the vanishing of local cohomology
  modules}, Invent. Math. \textbf{102} (1990), no.~1, 73--93. {\sf\scriptsize
  1069240 (91i:13020)}

\bibitem[Koh88]{KohSuperheight}
{\sc J.~Koh}: \emph{Super height of an ideal in a {N}oetherian ring}, J.
  Algebra \textbf{116} (1988), no.~1, 1--6. {\sf\scriptsize 944148 (89c:13021)}

\bibitem[Lyu02]{Lyubeznik-survey}
{\sc G.~Lyubeznik}: \emph{A partial survey of local cohomology}, Local
  cohomology and its applications ({G}uanajuato, 1999), Lecture Notes in Pure
  and Appl. Math., vol. 226, Dekker, New York, 2002, pp.~121--154.
  {\sf\scriptsize 1888197 (2003b:14006)}

\bibitem[Rob94]{Roberts-Example}
{\sc P.~C. Roberts}: \emph{A computation of local cohomology}, Commutative
  algebra: syzygies, multiplicities, and birational algebra ({S}outh {H}adley,
  {MA}, 1992), Contemp. Math., vol. 159, Amer. Math. Soc., Providence, RI,
  1994, pp.~351--356. {\sf\scriptsize 1266191 (95f:13022)}

\bibitem[SV86]{StruckradVogelBuchsbaumRings}
{\sc J.~St{\"u}ckrad and W.~Vogel}: \emph{Buchsbaum rings and applications},
  Mathematische Monographien [Mathematical Monographs], vol.~21, VEB Deutscher
  Verlag der Wissenschaften, Berlin, 1986, An interaction between algebra,
  geometry, and topology. {\sf\scriptsize 873945 (88h:13011b)}

\end{thebibliography}

\end{document}